\documentclass[a4paper,11pt]{article}

\usepackage{mathtools}
\usepackage{amsmath}
\usepackage{amssymb}
\usepackage{amsthm}
\usepackage{enumitem}
\usepackage{graphicx}
\usepackage{mathabx}
\usepackage{verbatim}
\usepackage{color}
\usepackage[%
ocgcolorlinks,%
linkcolor=blue,%
filecolor=blue,%
citecolor=blue,%
urlcolor=blue]{hyperref}
\usepackage{sidecap}
\usepackage{enumitem}

\usepackage[font={small,it}]{caption}

\newcommand{\bbL}{\mathbb{L}}

\newcommand{\bbM}{\mathbb{M}}
\newcommand{\bbE}{\mathbb{E}}
\newcommand{\bbV}{\mathbb{V}}

\newcommand{\Var}{\bbV{\rm ar}}

\newcommand{\bbP}{\mathbb{P}}

\newcommand{\bbZ}{\mathbb{Z}}
\newcommand{\bbR}{\mathbb{R}}

\newcommand{\eps}{\epsilon}

\newcommand{\ol}{\overline}

\newcommand{\wh}{\widehat}
\newcommand{\wc}{\widecheck}

\newcommand{\cA}{\mathcal A}

\newcommand{\cC}{\mathcal C}
\newcommand{\cX}{\mathcal X}

\newcommand{\cB}{\mathcal B}

\newcommand{\cW}{\mathcal W}

\newcommand{\cE}{\mathcal E}
\newcommand{\cN}{\mathcal N}
\newcommand{\cF}{\mathcal F}

\newcommand{\cQ}{\mathcal Q}
\newcommand{\cU}{\mathcal U}
\newcommand{\cV}{\mathcal V}
\newcommand{\cD}{\mathcal D}

\newcommand{\rmc}{{\rm c}}

\newcommand{\rmd}{{\rm d}}
\newcommand{\rme}{{\rm e}}
\newcommand{\rmB}{{\rm B}}

\newcommand{\wt}{\widetilde}

\newtheorem{theorem}{Theorem}[section]
\newtheorem{cor}[theorem]{Corollary}
\newtheorem{lem}[theorem]{Lemma}
\newtheorem{prop}[theorem]{Proposition}
\newtheorem{thm}[theorem]{Theorem}

\newtheorem{definition}[theorem]{Definition}
\newtheorem*{rem*}{Remark}

\numberwithin{equation}{section}
\graphicspath{{pic/}}

\title
{On the Growth of the Extremal and Cluster Level Sets in Branching Brownian Motion}
\date{}
\author
{Lisa Hartung \thanks{lhartung@uni-mainz.de,  oren.louidor@gmail.com, wtq0000@gmail.com} \\ 
Universit\"at Mainz\and Oren Louidor\footnotemark[1]\\Technion, Israel\and 
Tianqi Wu\footnotemark[1]\\Technion, Israel }

\begin{document}
\maketitle
\begin{abstract}
We study the limiting extremal and cluster point processes of branching Brownian motion. The former records the heights of all extreme values of the process, while the latter records the relative heights of extreme values in a genealogical neighborhood of order unity around a local maximum thereof. For the extremal point process, we show that the 
mass of upper level sets $[-v, \infty)$ grows as $C_\star Z v e^{\sqrt{2} v}(1+o(1))$ as $v \to \infty$, almost surely, where $Z$ is the limit of the associated  derivative martingale and $C_\star \in (0, \infty)$ is a universal constant. For the cluster point process, we show that the logarithm of the mass of $[-v, \infty)$ grow as $\sqrt{2}v$ minus random fluctuations of order $v^{2/3}$, which are governed by an explicit law in the limit. The first result improves upon the works of Cortines et al.~\cite{CHL19} and Mytnik et al.~\cite{MRR22} in which asymptotics are shown in probability, while the second makes rigorous the derivation 
in the physics literature by Mueller et al.~\cite{Munier1} and Le et al.~\cite{Munier2} and resolves a conjecture thereof.
\end{abstract} 

\tableofcontents

\section{Introduction and results}
\subsection{Setup}
This work addresses several questions which have been raised recently concerning the extreme values of branching Brownian motion (BBM). Let us first recall the definition of BBM and some of the state-of-the-art results concerning its extreme value statistics. Let $L_t$ be the set of particles alive at time $t \geq 0$ in a continuous time Galton-Watson process with binary branching at rate $1$. The entire genealogy can be recorded via the metric space $(T,\rmd)$, consisting of the elements $T := \bigcup_{t \geq 0} L_t$ and equipped with the {\em genealogical distance} 
\begin{equation}
\rmd (x,x') := \tfrac12 \Big(\big(|x|-|x \wedge x'|\big) + \big(|x'|-|x \wedge x'|\big)\Big)
\, , \qquad x,x' \in T \,.
\end{equation}
In the above, $|x|$ stands for the generation of $x$, namely $t$ such that $x \in L_t$ and $x \wedge x'$ is the most recent common ancestor of $x$ and $x'$. We shall also write $x_s$ for the ancestor of $x$ at generation $0 \leq s \leq |x|$.

Conditional on $(T,\rmd)$, let $h = (h(x) :\: x \in T)$ be a mean-zero continuous Gaussian process with covariance function given by 
\begin{equation}
	\bbE h(x) h(x') = |x \wedge x'| \, , \qquad x,x' \in T \,.
\end{equation}
The triplet $(h, T,\rmd)$ (or just $h$ for short) forms a standard BBM, and $h(x)$ for $x \in L_t$ is interpreted as the height of particle $x$ at time $t$. As such, this process models a system in which particles diffuse as standard Brownian motions (BMs) and undergo binary splitting at $\text{Exp}(1)$-distributed random times. The restriction of $T$ to all particles born up to time $t$ will be denoted by $T_t := \bigcup_{s \leq t} L_s$, with $\rmd_t$ and $h_t$ the corresponding restrictions of $\rmd$ and $h$, respectively. The natural filtration of the process $(\cF_t :\: t \geq 0)$ can then be defined via $\cF_t = \sigma(h_t, T_t, \rmd_t)$ for all $t \geq 0$.

The study of extreme values of $h$ dates back to works of Ikeda et al.~\cite{Ikeda1,Ikeda2,Ikeda3}, McKean~\cite{McKean}, Bramson \cite{B_M, B_C} and Lalley and Sellke~\cite{LS} who derived asymptotics for the law of the
maximal height $h^*_t = \max_{x \in L_t} h_t(x)$. Introducing the centering function
\begin{equation}
\label{eq_m_t}
m_t := \sqrt{2}t - \frac{3}{2\sqrt{2}} \log^+ t \,,
\quad  \text{where } \qquad
\log^+ t := \log (t \vee 1) \,,
\end{equation}
and writing
$\wh{h}_t$ for the centered process $h_t - m_t$ and $\wh{h}^*_t : = h_t^\ast - m_t$ for the centered maximum, these works show 
\begin{equation}
\label{e:101.4}
\wh{h}^*_t \underset{t \to \infty} \Longrightarrow G + \tfrac{1}{\sqrt{2}} \log Z  \,,
\end{equation}
where $G$ is a Gumbel distributed random variable with rate $\sqrt{2}$ and $Z$, which is independent of $G$, is the almost sure limit as $t \to \infty$ of (a multiple of) the so-called {\em derivative martingale}:
\begin{equation}
\label{e:303}
\textstyle
Z_t := C_\diamond \sum_{x \in L_t} \big( \sqrt{2} t - h_t(x) \big) \rme^{\sqrt{2} (h_t(x) - \sqrt{2}t)} \,,
\end{equation}
for some properly chosen $C_\diamond > 0$. Other extreme values of $h$ can be studied simultaneously by considering its {\em extremal process}. To describe the latter, given $t \geq 0$, $x \in L_t$ and $r > 0$, we let $\cC_{t,r}(x)$ denote the {\em cluster} of {\em relative heights} of particles in $L_t$, which are at genealogical distance at most $r$ from $x$. This is defined formally as the point measure
\begin{equation}
\label{e:5B}
\textstyle
\cC_{t,r}(x) := \sum_{y \in \rmB_{r}(x)} \delta_{h_t(y) - h_t(x)},\  \text{ where} \quad
\rmB_{r}(x) := \{y \in L_t :\: \rmd (x,y) < r\} \,.
\end{equation}

Fixing any positive function $[0,\infty) \ni t \mapsto r_t$ such that both $r_t$ and $t-r_t$ tend to $\infty$ as $t \to \infty$ and letting
\begin{equation}
	L_t^* = \big\{x \in L_t :\: h_t(x) \geq h_t(y) \,, \forall y \in \rmB_{r_t} (x)\big\} \,,
\end{equation} 
the {\em structured extremal process} is then given as
\begin{equation}
\label{e:N6}
\textstyle
\wh{\cE}_t := \sum_{x \in L_t^*} \delta_{h_t(x) - m_t} \otimes \delta_{\cC_{t,r_t}(x)} \,.
\end{equation}
That is, $\wh{\cE_t}$ is a point process on $\bbR \times \bbM_p((-\infty,0])$, where $\bbM_p((-\infty,0])$ denote the space of  Radon point measures on $(-\infty, 0]$, which records {\em the-centered-heights-of} and {\em the-clusters-around} all $r_t$-local maxima of $h$. Then, it was (essentially) shown in~\cite{ABBS2013, ABK_E} that
\begin{equation}
\label{e:N7}
\big(\wh{\cE}_t, Z_t \big) \underset{t \to \infty}{\Longrightarrow} \big(\wh{\cE}, Z\big)\,,
\quad \text{where}\qquad 
\wh{\cE} \sim {\rm PPP}(Z\rme^{-\sqrt{2}u} \rmd u \otimes \nu) \,.
\end{equation}
In the above, $Z_t$ and $Z$ are as before and $\nu$ is a deterministic distribution on $\bbM_p((-\infty,0])$, which we will call the {\em cluster distribution}. The law of $\wh{\cE}$ should be interpreted as that of a point process realized by first drawing $Z$ and then, conditional on its value, sampling a PPP with intensity $Z\rme^{-\sqrt{2}u} \rmd u \otimes \nu$. Also, the weak convergence of the first coordinate above is w.r.t.~the vague topology on the space of Radon measures on $\bbR \times \bbM_p((-\infty,0])$, which is, in turn, equipped with the product of the Euclidean metric and the vague metric on $\bbM_p((-\infty,0])$.

As two consequences, one gets the convergence of the {\em extremal process of local maxima}:
\begin{equation}
\cE^*_t \underset{t \to \infty}\Longrightarrow \cE^*  \,,
\end{equation}
where
\begin{equation}
\label{e:5.5}
\cE_t^* := \sum_{x \in L^*_t} \delta_{h_t(x) - m_t} 
\quad , \qquad 
\cE^* := \sum_{(u, \cC) \in \wh{\cE}} \, \delta_u 
\,\sim\, {\rm PPP}(Z\rme^{-\sqrt{2}u} \rmd u) \,,
\end{equation}
and the convergence of the {\em usual} extremal process of $h$:
\begin{equation}
\cE_t \underset{t \to \infty}\Longrightarrow \cE  \,,
\end{equation}
where
\begin{equation}
\label{e:O.2}
\cE_t := \sum_{x \in L_t} \delta_{h_t(x) - m_t} 
\, , \qquad 
\cE
:= \sum_{k \geq 1} \, \cC^k(\cdot - u^k) \,,
\end{equation}
in which $u^1 > u^2 > \dots$ enumerate the atoms of $\cE^*$ and $(\cC^k)_{k \geq 1}$ are i.i.d.~chosen according to $\nu$ and independent of $\cE^*$. 

Since, conditional on $Z$, the intensity measure of $\cE^*$ is finite on $[-v, \infty)$ for every $v \in \bbR$ and tends to $\infty$ as $v \to \infty$, it is a standard fact that $\bbP(-|Z)$-a.s. and therefore also $\bbP$-a.s.,
\begin{equation}
\label{e:1.3}
	\frac{\cE^*([-v, \infty))}{\bbE \big(\cE^*([-v, \infty))\,\big|\,Z\big)} 
	= \frac{\cE^*([-v, \infty))}{\tfrac{1}{\sqrt{2}} Z \rme^{\sqrt{2}v}} 
	\underset{v \to \infty}{\longrightarrow} 1 \,.
\end{equation}

Asymptotics for $\cE([-v, \infty))$, which is arguably a more interesting quantity, is not however a straightforward consequence of~\eqref{e:5.5} and~\eqref{e:O.2}. This is because the limiting process $\cE$ is now a superposition of i.i.d.~clusters $\cC$, and the law $\nu$ of the latter will determine the number of points inside any given set in the overall process. To address this question, a study of the moments of the level sets of a cluster $\cC \sim \nu$ was carried out in~\cite{CHL19}. In particular, it was shown that  
as $v \to \infty$,
\begin{equation}
\label{e:29}
\bbE \cC([-v, 0]) = C_\star \rme^{\sqrt{2} v} (1+o(1)) \,.
\end{equation}
for some $C_\star > 0$. This was then combined with~\eqref{e:5.5} and~\eqref{e:O.2} to derive 
\begin{equation}
\label{e:1.10}
\frac{\cE([-v, \infty))}{C_\star Z v \rme^{\sqrt{2} v}} \overset{\bbP}{\underset{v \to \infty}{\longrightarrow}} 1 \,.
\end{equation}
This result was re-derived in~\cite{MRR22} using PDE techniques, which in addition allows one to identify the precise value of $C_\star$. Notice that compared to the asymptotics for the level sets of $\cE^*$, there is an additional linear prefactor (in $v$) in that of $\cE$, which is due to the contribution of the clusters.

\subsection{Results}
As the convergence in~\eqref{e:1.10} is in probability, a natural question raised following the appearance of~\cite{CHL19} was whether this can be strengthened to 
almost sure convergence, as is the case in~\eqref{e:1.3}. Our first result answers this question affirmatively.
\begin{thm}\label{1st_result} 
With $C_\star$ as in~\eqref{e:1.10}, almost surely,
	\begin{equation}\label{eq:1st_result}
	\frac{\cE([-v, \infty))}{C_\star Z v e^{\sqrt{2} v}} \underset{v \to \infty}\longrightarrow 1 \,.
	\end{equation}
\end{thm}

Another question that was raised concerns the typical asymptotic growth of the size of the cluster level sets $\cC([-v, 0])$. As was pointed out in~\cite{CHL20}, the first moment asymptotics in~\eqref{e:29} {\em do not} capture the asymptotic typical size of level set $[-v, 0]$, as the mean is the result of an unusually large level set size, occurring with a small (vanishing with $v$) probability. Nevertheless, no study of this question beyond moment analysis was carried out in that work.

This problem was addressed in the physics literature by~\cite{Munier2, Munier1}, where based on simulations and heuristic arguments the authors predicted that 
\begin{equation}
	\cC([-v,0]) \approx \rme^{\sqrt{2}v - \Theta(v^{2/3})} \,,
\end{equation}
where $\Theta(v^{2/3})$ is a random quantity of order $v^{2/3}$ whose law (scaled by $v^{2/3}$) was roughly conjectured, and $\approx$ means that the logs of both sides are asymptotically equivalent as $v \to \infty$. Our second result proves this conjecture and identifies the law of this random quantity.
\begin{thm}\label{2nd_result} Let $\cC \sim \nu$. There exists an almost surely positive and finite random variable $\zeta$ such that
	\begin{equation}\label{eq:2nd_result}
		\frac{\log \cC([-v, 0]) -\sqrt{2} v }{v^{\frac{2}{3}}} \underset{v \to \infty} \Longrightarrow -\zeta \,.
	\end{equation}
Moreover, the limit $\zeta$ may be realized as
\begin{equation}\label{chi_def}
	\zeta := \inf_{s > 0} \Big(\sqrt{2}\, Y_s + \frac{1}{2s}\Big) \,,
\end{equation}
where $(Y_s)_{s \geq 0}$ is a Bessel-3 process starting from $0$.
\end{thm}
We remark that the law of $\zeta$ differs from its conjectured form in~\cite{Munier2}. This comes from an intrinsic fine scale optimization problem which was not visible in the latter work.
As an immediate corollary, one gets
\begin{cor}
Let $\cC \sim \nu$.
\begin{equation}
	\frac{\log \cC([-v, 0])}{\sqrt{2} v} \overset{\bbP}{\underset{v\to\infty}{\longrightarrow}}
	1 \,.
\end{equation}
\end{cor}
In fact, a straightforward modification of the proof of Theorem~\ref{2nd_result} gives
\begin{prop}
\label{p:2}
Let $\cC \sim \nu$. Then, almost surely, 
\begin{equation}
	\frac{\log \cC([-v, 0])}{\sqrt{2} v} {\underset{v\to\infty}{\longrightarrow}}
	1 \,.
\end{equation}
\end{prop}

In order to prove Theorem~\ref{2nd_result}, we develop a representation of the cluster law $\nu$ in which the cluster $\cC$ is obtained as the \emph{almost sure} limit of certain point processes, all defined on the same space. This is in contrast to the representation in~\cite{CHL19,CHL20} (as well as that in~\cite{ABK_E}), in which the law $\nu$ is obtained by taking a \emph{weak} limit of objects, which are defined on different underlying spaces, and with respect to a conditional measure, where the conditioning becomes singular in the limit. Moreover, the point processes involved in our representation are defined in terms of well-controlled objects: 
\begin{itemize}
	\item a Bessel-3-like ``backbone'' process $\wc{\cW}$,
	\item a Poisson-like point process of time stamps $\wc{\cN}$,
	\item a collection $\wc{\cE}$ of extremal processes for independent BBM conditioned on their maximum staying below a typical value.
\end{itemize}  
We shall henceforth refer to this new representation as {\em strong}, and to the one in~\cite{CHL19,CHL20} as {\em weak}, noting that this designation refers to the type of limit taken in it.

In what follows, $C([0, \infty))$ is the space of continuous functions on $[0, \infty)$ endowed with the topology of uniform convergence on compact sets. Also, given a function $f$ and a subset $A$ of its domain, we write $f_A$ for the restriction of $f$ to $A$.
\begin{thm}
\label{p:1} 
There exists a process
$(\wc{W}, \wc{\cN}, \wc{\cE})$
taking values in 
\begin{equation}
	C([0, \infty)) \times \bbM_p([0, \infty)) \times \bbM_p((-\infty,0])^{\bbR_+\times \bbR} \,,
\end{equation} 
such that for
\begin{equation}
	\wc{\cC}_r := \int_0^r \wc{\cE}_{s, \wc{W}_s} \, \wc{\cN}(\rmd s) \,,
\end{equation}
the limit 
\begin{equation}\label{eq:C_limit_r_to_infty}
	\wc{\cC} := \lim_{r \to \infty} \wc{\cC}_r
	\equiv \int_0^\infty \wc{\cE}_{s, \wc{W}_s} \, \wc{\cN}(\rmd s)
\end{equation}
holds pointwise (and therefore vaguely) almost surely and obeys 
\begin{equation}
	\wc{\cC} \sim \nu \,,	
\end{equation}
where $\nu$ is the cluster law. Moreover, the following holds:
\begin{enumerate}   
\item $(\wc{\cE}_{s,y})_{s,y}$ are independent and also independent of $(\wc{W}, \wc{\cN})$, and for all $s \geq 0$, $y \in \bbR$, 
    \begin{equation}
    \label{e:101.12}
	\wc{\cE}_{s,y} \overset{\rmd}= \cE_s(-y + \cdot) \,\big|\, \big\{\cE_s((-y, \infty)) = 0\big\}\,,
\end{equation}
where $\cE_s$ is as in~\eqref{e:O.2}. 

\item For $(Y_s)_{s \geq 0}$ a Bessel-3 process, $\wc{Y}_s := -Y_s - \frac{3}{2\sqrt{2}} \log^+s$, and any $y_0 \geq 0$,
\begin{equation}
\label{e:2.3}\
	\Big\|\bbP \big(\wc{W}_{[r, \infty)} \in \cdot) - \bbP \big(\wc{Y}_{[r, \infty)} \in \cdot\,\big|\, \wc{Y}_0 = y_0 \big) \Big\|_{\rm TV} \underset{r \to \infty} \longrightarrow 0 \,.
\end{equation}
\item There exists $C,c \in (0,\infty)$ and for all $\epsilon > 0$ also $C, c' \in (0,\infty)$ such that for all $s,u \geq 0$,
\begin{equation}
\label{e:101.14}
	\bbP\big(\inf \big\{|\sigma - s| :\: \sigma \in \wc{\cN}\big\} > u\big) \leq Ce^{-cu} 
\quad , \qquad
\bbP \big(\wc{\cN}([0, s]) > (2+\epsilon) s) \leq C'e^{-c's} \,.
\end{equation}
\end{enumerate}
\end{thm}

We remark that this strong representation of the cluster law is essentially equivalent to that in Theorem~2.3 in~\cite{ABBS2013}, in which the joint law of $(\wc{W}, \wc{\cN}, \wc{\cE})$ is explicitly specified. In particular, the law of $\wc{W}$ is expressed there as a measure change (involving the right tail function for the centered maximum) of a certain process which is a BM until some stopping time and then Bessel-3 from that point on.
In contrast, in our result, while we do not express the law of the underlying process explicitly, we provide key properties of its ingredients. We believe that for certain purposes, e.g. Theorem~\ref{2nd_result}, this more qualitative cluster representation would be more convenient.

The most important property in this theorem is the asymptotic equivalence in law, up to a deterministic logarithmic curve, between $\wc{W}$ and $-Y$, a negative Bessel-3 process. The underlying topology for this convergence is defined by the Total Variation distance between probability measures on the space of continuous trajectories which are unbounded in time. This is a very strong sense, which allows one to use infinite-time properties of Bessel-3 to derive similar properties for the cluster distribution.

Our derivation of Theorem~\ref{p:1} does not use~Theorem~2.3 in~\cite{ABBS2013}, but rather relies on the weak representation of the cluster law from~\cite{CHL19,CHL20}. An alternative approach for obtaining Bessel-3 statistics for the backbone process $\wc{W}$ was taken recently in~\cite{KimZeit}, in which the authors analyze the explicit law of this process, as given in~\cite{ABBS2013}. This was done as part of studying the extremal structure of BBM in dimension $d \geq 2$ (see also~\cite{KLZ23,BKLMZ24}). Their analysis indeed shows convergence to Bessel-3, but in a somehow weaker form: $(L^{-1}\wc{W}_{L^2s})_{s \in [0,N]} \Longrightarrow -Y_{[0,N]}$ as $L \to \infty$ for any $N < \infty$. 

\subsection{Proof Outline}
\label{s:1.1}
Let us briefly discuss the idea behind the proofs of Theorem~\ref{1st_result} and Theroem~\ref{2nd_result}.
Starting with the first of the two, recall that $\cE$ is obtained as a superposition of i.i.d.~clusters whose law is $\nu$ and whose tips follow the atoms of $\cE^*$, where $\cE^*$ is a PPP with an exponential density that is randomly shifted by $\frac{1}{\sqrt{2}} \log Z$. The proof of Theorem~\ref{1st_result} goes by separately controlling the contribution to $\cE([-v, \infty))$ from clusters whose tip is above and below $u = -\delta \log v$. Here $\delta$ can be any positive real number smaller than $1/\sqrt{2}$. 

For the contribution coming from tips below $-\delta \log v$, denoted $\cE ([-v, \infty) ;\; [-v, -\delta \log v])$, we can simply improve the moment analysis from~\cite{CHL19}, to yield a.s.~convergence via Borel-Cantelli,
\begin{equation}
\lim_{v\rightarrow \infty} \frac{\cE\big([-v, \infty) ; \; [-v, -\delta \log v] \big)}{v e^{\sqrt{2} v}} = C_\star Z \,.
\end{equation}
(See Lemma~\ref{l:3.1} and the proof of Theorem \ref{1st_result}.)

On the other hand, for $\cE ([-v, \infty) ;\; (-\delta \log v, \infty))$, i.e.~the contribution from clusters whose tip is above $-\delta \log v$, we show that almost surely
	\begin{equation}
	\label{e:101.30a}
		\lim_{v\rightarrow \infty} \frac{\cE\big([-v, \infty) ; \; [-\delta \log v, \infty) \big)}{v^{-K} \rme^{\sqrt{2} v}} = 0\,,
	\end{equation}
for any $K > 0$. Here the analysis is more delicate and cannot be done directly via control of its moments, as this quantity is not concentrated enough. 

Indeed, the typical contribution of clusters whose tip is $w \in (-\delta \log v, 0]$ is 
\begin{equation}
\label{e:101.31}
\cC([-u, 0]) = o(\rme^{\sqrt{2}u}) = o(\rme^{\sqrt{2}v}) 
\quad,\qquad  u = v+w \,.
\end{equation}
This can be seen, e.g.~from Theorem~\ref{2nd_result}. However, as was shown in~\cite{CHL20}, for a given $u$ the probability that this typical event occurs is ``only'' as high as $1-\Theta(u^{-1})$ and with the complement probability, the contribution is atypically as high as
\begin{equation}
\cC([-u, 0]) = \Theta(u\rme^{\sqrt{2}u}) \,.
\end{equation}
This is not enough to guarantee~\eqref{e:101.30a}, even with $K=-1$. Fortunately, with only slightly less probability ($1-O(u^{-1}(\log u)^{5/2})$), one can ensure that~\eqref{e:101.31} holds simultaneously for all $u \geq u_0$ (Proposition~\ref{higher_leaders}). This is shown by appealing to the weak cluster law representation from~\cite{CHL19,CHL20} as precise quantitative estimates are needed. Thanks to the controlled exponential growth of the cluster tips in $\cE^*$, this in turn implies~\eqref{e:101.30a}.

Turning to the proof of Theorem~\ref{2nd_result}, the starting point here is the convenient strong representation of the cluster law as given by Theorem~\ref{p:1}.
Using (by now standard) truncated first and second moment methods, one can fairly easily show that with high probability as $s \to \infty$,
\begin{equation}
\cE_s([-v, \infty)) =  e^{\sqrt{2} v - \frac{v^2}{2s} + O(\log s)} \,,
\end{equation}
uniformly in  $v \leq s^{1-o(1)}$ (Lemma~\ref{1st_moment} and Lemma~\ref{lower_bound_2}). While this was mostly shown for $v = O(1)$ in the past, the proofs carry over to these {\em deep} extreme level sets (in the precision required) with a properly modified truncation.

Now, thanks to Part~2 of Theorem~\ref{p:1}, with high probability the backbone $\wc{W}$ behaves like the negative of a Bessel-3 process $Y$ plus a logarithmic curve after some large time. Because of the diffusive scale of Bessel-3, this shows that $-\wc{W}_s = Y_s + O(\log s) = \Theta(\sqrt{s})$ for large $s$, and thus the conditioning in~\eqref{e:101.12} becomes superfluous. Together with the fact that  the timestamp process $\wc{\cN}$ is reasonably behaved (i.e.~\eqref{e:101.14}), we thus get from~\eqref{eq:C_limit_r_to_infty}
\begin{equation}
\label{e:101.30}
\begin{split}
\wc{\cC} & = \int_0^\infty \rme^{\sqrt{2} (v + \wc{W}_s) - \frac{(v+\wc{W}_s)^2}{2s} + O(\log s)}
\wc{\cN}(\rmd s)
\approx \rme^{\sqrt{2} v}
\int_0^\infty \rme^{ -\sqrt{2} Y_s - \frac{v^2}{2s} + \frac{v}{s} Y_s +O(\log s)} \rmd s \\
& \approx 
\rme^{\sqrt{2} v} v^{O(1)} \exp \Big(-\min_s \Big(\sqrt{2} Y_s + \frac{v^2}{2s} - O\big(\tfrac{v}{\sqrt{s}} \big)\Big)\Big)
\approx
\rme^{\sqrt{2} v} v^{O(1)} \exp \Big(-\min_s \Big(\Theta(\sqrt{s}) + \frac{v^2}{2s}\Big)\Big) \,.
\end{split}
\end{equation}
The last minimum is attained at $s_* = \Theta(v^{4/3})$, at which its value is $\Theta(v^{2/3})$. This gives the $\Theta(v^{2/3})$ fluctuations in the exponent for the typical cluster level set size. 

To obtain convergence in law, we focus on the scale $s = \Theta(v^{4/3})$ by 
employing the rescaling
\begin{equation}
	\hat{s} := v^{-\frac{4}{3}} s \,,
 \qquad			Y^{(v)}_{\hat s} := v^{-\frac{2}{3}} Y_{s} \,.
\end{equation}
Then the second to last minimization (restricted to $s = \Theta(v^{4/3}$)) is (up to errors of $O(v^{1/3})$)
\begin{equation}
v^{\frac{2}{3}} \min_{\hat{s}} \Big(\sqrt{2} Y^{(v)}_{\hat{s}}  + \frac{1}{2\hat{s}}\Big)  \,.
\end{equation}
The minimum here identifies in law with $\zeta$ in~\eqref{chi_def}, again thanks to the Brownian scaling invariance of Bessel-3.

\paragraph{Organization of the paper}
The remainder of the paper is organized as follows. Section~\ref{s:2} includes preliminary results which are needed for the proofs to follow. These include standard facts about Bessel process and Ballot Theorems for Brownian motion, as well as an account of the weak representation of the cluster law from~\cite{CHL19, CHL20}. This section also lays down any general notation used throughout. The proofs of Theorem~\ref{1st_result}, Theorem~\ref{2nd_result} and Theorem~\ref{p:1} are given in Sections~\ref{s:3},~\ref{s:5} and~\ref{s:4} respectively. Section~\ref{s:5} also includes the short proof of Proposition~\ref{p:2}.

\section{Preliminaries}
\label{s:2}
In this section we include preliminary statements and notation which will be used throughout this manuscript. 

\subsection{Additional notation}
For a Markovian process $Z = (Z_t)_{t \geq 0}$, we shall formally, but conventionally, write $\bbP(-|Z_0=x)$ to denote the case when $Z$ is defined to have $Z_0 = x$. When there is no ambiguity concerning the underlying process, we shall abbreviate:
\begin{equation}
\label{e:102.1}
	\bbP(-) \equiv \bbP(-\,|\,Z_0 = 0\big)
	\ , \quad 
	\bbP_{r,x}(-) \equiv \bbP(-\,|\,Z_{r} = x\big) 
	\  , \quad
	\bbP_{r,x}^{s,y}(-) \equiv \bbP(-\,|\,Z_{r} = x, Z_{s} = y) \,.
\end{equation}
Throughout this manuscript, $W = (W_s)_{s \geq 0}$ will always denote a standard Brownian motion. If $\mu$ is a measure on some measure space $(\cX, \Sigma)$, and $f: \cX \to \bbR_+$, we shall write $\mu(\rmd x) f(x)$ to denote the measure whose Radon-Nikodym derivative w.r.t.~$\mu$ is $f$. Finally, we write $a \lesssim b$ if there exists $C \in (0,\infty)$ such that $a \leq C b$ and $a \asymp b$ if both $a \lesssim b$ and $b \lesssim a$. To stress that $C$ depends on some parameter $\alpha$ we shall add a subscript to the relation symbols. As usual $C$, $c$, $C'$, etc. denote positive and finite constants whose value may change from one use to another.

\subsection{Bessel-3 and Ballot Theorem}
In this subsection we include known statements about Brownian motion conditioned to stay positive and the closely related Bessel-3 process. 

\subsubsection{The Bessel-3 Process}
Recall that a 3-dimension Bessel Process $Y \in (Y_s)_{s\geq 0}$ starting from $Y_0 = y_0$ can be defined as the norm of a 3 dimensional Brownian motion $\vec{W} = (W^{(1)},W^{(2)},W^{(3)})$ starting from $\|\vec{W}_0\| = y_0$, so that
\begin{equation}
\label{e:202.2}
\bbP \big( (Y_s)_{s \geq 0} \in \cdot \,\big|\, Y_0 = y_0 \big) = 
	\bbP \big((\|\vec{W}_s\|)_{s \geq 0} \in \cdot \,\big|\, \|\vec{W}_0\| = y_0 \big) \,.
\end{equation}
It follows that $Y$ inherits the scaling invariance of BM, namely for any $a > 0$ and $y_0 \geq 0$,
\begin{equation}
\label{e:202.4}
	\bbP \big((a^{-1/2} Y_{as})_{s \geq 0} \in \cdot\,\big|\, a^{-1/2} Y_0 = y_0\big) 
	 = \bbP \big(Y \in \cdot\,\big|\, Y_0 = y_0\big) \,.
\end{equation}

For $y_0 > 0$, an alternative definition for $Y$ can be given in terms of Doob's $h$-transform of a standard BM via the function $h(x) = x$. More precisely, $Y$ is a $C([0,\infty))$-valued process which satisfies for all $r > 0$,
\begin{equation}
\label{e:202.3}
	\bbP \big(Y_{[0,r]} \in \rmd y\,\big|\, Y_0 = y_0 \big) = \bbP\big(W_{[0,r]} \in \rmd y\,\big|\,W_0 = y_0\big) \frac{y_r}{y_0} 1_{\{\min y_{[0,r]} > 0\}}
	\,.
\end{equation}
See, for example,~\cite{McKean}. In particular, the distribution of $Y_{[0,r]}$ is absolutely continuous w.r.t. to the (sub-probability) distribution of $W_{[0,r]}$ restricted to $C([0,r], \bbR_+)$, both with the same initial conditions $y_0 > 0$.

The next lemma shows that a Bessel process forgets its initial (space/time) conditions. While this is a standard result, the following is a strong ``infinite-horizon'' version which uses the Total Variation norm.
\begin{lem}
\label{l:Bessel_TV}
Let $x, y\geq 0$ and $s\geq 0$. Then
\begin{equation}
	\Big\|\bbP(Y_{[r, \infty)} \in \cdot \,\big|\, Y_0 = x \big) - 
		\bbP(Y_{[s+r, \infty)} \in \cdot \,\big|\, Y_0 = y \big) \Big\|_{\rm TV}
		\underset{r \to \infty}\longrightarrow 0 \,.
\end{equation}
Moreover, any fixed $R > 0$ and $s \geq 0$, the rate of convergence is uniform for $0 \leq x, y \leq R$.
\end{lem}
\begin{proof} Let us first prove the case $s = 0$. Let $Y^{(x)}$ and $Y^{(y)}$ be independent Bessel processes started at $x$ and $y$, and
\begin{equation}
T_{x, y} := \inf_{t\geq 0} \left\{t: Y^{(x)}_t = Y^{(y)}_t\right\}.
\end{equation}
By using strong Markov property and a straightforward coupling, it is easy to see the stochastic domination $T_{x, y} \leq_s T_{0, R}$. Now, by coupling $Y^{(x)}$ and $Y^{(y)}$ after time $T_{x, y}$, we see that
\begin{equation}
d_{\rm TV}\left(Y^{(x)}_{[r, \infty)}, Y^{(y)}_{[r, \infty)}\right) \leq \bbP(T_{x, y} > r) \leq \bbP(T_{0, R} > r) \to 0 \,,
\end{equation}
as $r \to \infty$ because $T_{0, R} < \infty$ almost surely. (To see the latter, observe that by scale invariance and the strong Markov property of $Y$, the probability $p := \bbP(T_{0, R} = \infty)$ does not depend on $R$ and satisfies $p \leq qp$, where $q = \bbP(T_{0,R} > 1) < 1$.)

Turning to the case $s > 0$, by the Markov property again, 
\begin{align}
d_{\rm TV}\left(Y^{(x)}_{[r, \infty)}, Y^{(y)}_{[s+r, \infty)}\right) &\leq \bbP(Y^{(y)}_s > R + R') + \max_{y' \leq R + R'} d_{\rm TV}\left(Y^{(x)}_{[r, \infty)}, Y^{(y')}_{[r, \infty)}\right) \\ 
& \leq \bbP(Y^{(R)}_s > R + R') + \bbP(T_{0, R+R'} > r) \to 0 \,,
\end{align}
as $r \to \infty$ followed by $R' \to \infty$.  
\end{proof}

It is well-known that the sample paths of Bessel-3 are almost surely $\big(\frac{1}{2}-\epsilon\big)$-H\"older continuous, which directly follows from (via~\eqref{e:202.3}) the same property for BM.
\begin{lem}
\label{uniform_continuity} For any $y_0 \geq 0$ and $\eps, K > 0$, there exists $\Lambda(\eps, K, y_0) > 0$ such that with probability at least $1 - \eps$, under $\bbP(-|Y_0=y_0)$ the process 
$Y$ is $(\frac{1}{2}-\eps)$-H\"older on the interval $[0, K]$ with H\"older constant at most $\Lambda(\eps, K, y_0)$.
\end{lem}  

For the typical envelope of sample paths of $Y$ we have,
\begin{lem}
\label{D-E} For any $y_0 \geq 0$ and $\eps > 0$, there exists $K(\eps,y_0)$ such that with $\bbP(-|Y_0=y_0)$
 probability at least $1 - \eps$,
	\begin{equation}
	s^{\frac{1}{2}-\eps} \leq Y_s \leq 3(s \log \log s)^{\frac{1}{2}} \,,
	\end{equation}
	for all $s \geq K(\eps, y_0)$. 
\end{lem} 
\begin{proof}
Employing the representation of Bessel-3 in~\eqref{e:202.2}, the lower bound follows immediately from the Dvoretzky-Erd\"os Test (c.f.~\cite{PeresMorters}) and the upper bound by the Law of Iterated Logarithms applied to each component of $\vec{W}$ separately and the Union Bound (or, more sharply, to $\|\vec{W}\|$, c.f.~\cite{motoo1959proof}).

\end{proof}

\subsubsection{Ballot Estimates for BM}
It is also well known that the Bessel-3 process arises when one (formally) conditions a one-dimensional BM to stay positive forever. Let us first address the finite-time version of this positivity restriction. The following is well known and also easy to prove by the Reflection Principle.
\begin{lem}
\label{l:102.1}
For all $x \geq 0$, $y \geq 0$ and $t > 0$,
\begin{equation}
	\bbP\Big(\min_{s \in [0,t]} W_s \geq 0 \,\big|\, W_0 = x,\, W_t = y\Big) \leq 
	\frac{2xy}{t}\,.
\end{equation}
Moreover,
\begin{equation}
	\bbP\Big(\min_{s \in [0,t]} W_s \geq 0 \,\big|\, W_0 = x,\, W_t = y\Big) = 
	\frac{2xy}{t}(1+o(1))
\end{equation}
where $o(1) \to 0$ as $t \to \infty$, uniformly in $x,y \geq 0$ such that $xy \leq t^{1-\epsilon}$ for any fixed $\epsilon > 0$.
\end{lem}

Turning to the infinite time version, the next lemma shows convergence to Bessel-3 in a strong sense, namely under the Total Variation norm.
\begin{lem}
\label{l:2.2}
	For any $x,y > 0$ and $r > 0$, and any $(d_t)_{t\geq 0}$ satisfying $d_t \to 0$ as $t\to \infty$,  
	\begin{equation}\label{eq:Bessel_approx_TV}
		\Big\|(W_s + d_t s)_{s\in [0,r]} \in \cdot \,\big|\, W_0 = x, W_t = y, \min_{s \in [0,t]} W_s \geq 0 \big) - 
		\bbP(Y_{[0,r]} \in \cdot \,\big|\, Y_0 = x \big) \Big\|_{\rm TV}
		\underset{t \to \infty}\longrightarrow 0 \,.
	\end{equation}
\end{lem}
\begin{proof}
For $w \in C([0, r])$ and $b\in \bbR$, we denote $w^{(b)}_s = w_s + b s$ for brevity.  
 Using Bayes' rule, the Markov property, and the Cameron-Martin formula, we have
	\begin{align}\label{eq:BM_conditioned_1}
		\bbP_{0, x}^{t, y} \left(W^{(d_t)}_{[0,r]} \in \rmd w \, \Big| \min_{s \in [0,t]} W_s \geq 0 \right) &= \bbP_{0, x} \left(W^{(d_t)}_{[0,r]} \in \rmd w \right) \cdot 1_{\{\min w^{(-d_t)}_{[0,r]}\geq 0\}} \cdot g^{(1)}_{t, r}(w^{(-d_t)}) \\
  &= \bbP_{0, x} \left(W_{[0,r]} \in \rmd w\right) \cdot 1_{\{\min w^{(-d_t)}_{[0,r]}\geq 0\}} \cdot g^{(1)}_{t,r}(w^{(-d_t)}) \cdot g^{(2)}_{t,r} (w)
	\end{align}	
with
\begin{equation}\label{eq:R-N-for-stay-positive}
g^{(1)}_{t,r}(w) := \frac{p_{t-r}(w_r, y)}{p_t(x, y)} \cdot  \frac{\bbP^{t, y}_{r, w_r}(\min W_{[r,t]} \geq 0)}{\bbP^{t, y}_{0, x}(\min W_{[0, t]} \geq 0)} \,,
\end{equation}
where $p_t(x, y)$ is the transition kernel for Brownian motion and
\begin{equation}
g^{(2)}_{t,r}(w) = \exp \big(d_t w_r - \frac{1}{2} d_t^2 r\big) \,,
\end{equation}
which converges to $1$ as $t \to \infty$ for all fixed $w$ and bounded by an exponential function in $w_r$ because $d_t \to 0$ as $t\to \infty$. 

By Lemma~\ref{l:102.1}, we have that as $t\to \infty$ for fixed $x,y,z$,
\begin{equation}
\frac{\bbP^{t, y}_{r, z}(\min W_{[r,t]} \geq 0)}{\bbP^{t, y}_{0, x}(\min W_{[0, t]} \geq 0)} = \frac{z}{x} (1 + o(1))
\end{equation}
uniformly for $z$ in any compact set of $[0, \infty)$ and with the right hand side an upper bound up to a constant factor. Moreover, by direct computation,
\begin{align}
\frac{p_{t-r}(z, y)}{p_t(x, y)} &= \frac{1}{\sqrt{1 - \frac{r}{t}}} \exp\left(-\frac{(z-y)^2}{2(t-r)} + \frac{(y-x)^2}{2t}\right) =  1 + o(1) \,,
\end{align}
uniformly for $z$ in any compact set and with the quantity on the right hand side bounded by a constant for all $t \geq 1$. 
Combining these estimates and the fact that $d_t \to 0$ as $t\to \infty$, we have
\begin{equation}\label{eq:BM_Bessel_convergence_final}
1_{\{\min w^{(-d_t)}_{[0,r]}\geq 0\}} \cdot g^{(1)}_{t,r}(w^{(-d_t)}) \cdot g^{(2)}_{t,r} (w) \to 1_{\{\min w_{[0,r]}\geq 0\}} \frac{w_r}{x}
\end{equation}
as $t\to \infty$ whenever $\min w_{[0,r]}\neq 0$ (which happens $\bbP_{0,x}$ almost surely) and bounded by an exponential function in $w_r$ (which is $\bbP_{0,x}$-integrable). It follows by Dominated Convergence theorem that the convergence in~\eqref{eq:BM_Bessel_convergence_final} is also $L^1(\bbP_{0,x})$. This proves the result since the right hand side is also the density (w.r.t.~BM started at $x$) of Bessel-3, as stated in~\eqref{e:202.3}.
\end{proof}

\subsection{Weak representation of the cluster law}
\label{s:2.1}
In this subsection we include a subset of the theory that was developed in~\cite{CHL19} in order to get a handle on the law of the clusters. The analysis of the cluster distribution is done in three steps, which correspond to the next three sub-sections. In the first step one expresses the law of the clusters as the weak limit as $t \to \infty$ of the relative heights of the genealogical neighbors of a distinguished particle, called the ``spine'', which is conditioned to be the global maximum and at a prescribed height. Here one appeals to the so-called spine decomposition method (or spine method, for short). The second step is to recast events involving the spine being a global maximum, as events involving a certain ``decorated random-walk-like process'' which is required to stay below a given curve. We refer to such events as ``Ballot'' events, on account of the analogy with a simple random walk which is restricted to be positive. In the third step one estimates the probability of such Ballot events. This is achieved by appealing to Ballot estimates for standard Brownian motion.

\subsubsection{Spinal decomposition}
The {\em $1$-spine branching Brownian motion} (henceforth sBBM)  evolves like BBM, except that at any given time, one of the living particles is designated as the ``spine''. What distinguishes this particle from the others, is that the spine branches at Poissonian rate $2$ and not $1$ as the remaining living particles. There is no difference in the motion, which is still that of a BM. When the spine branches, one of its two children is chosen uniformly at random to become the new spine. 

We shall keep the same notation $(h_t, T_t, \rmd_t)_{t \geq 0}$ for the genealogical 
and positional processes associated with this model. The additional information, namely the identity of the spine at any given time, will be represented by the process $(X_t)_{t \geq 0}$, where $X_t \in L_t$ for all $t \geq 0$. We shall also use the same notation as in the case of BBM to denote objects associated with the motion $h$, e.g. $\wh{h}_t$ will still denote $h_t - m_t$. To make the distinction from regular BBM explicit, we will denote the underlying probability measure by $\wt{\bbP}$ (and $\wt{\bbE}$), in accordance with the notation in~\cite{CHL19}. 

A useful representation for this process can be obtained by ``taking the point of view of the spine''. More precisely, one can think of sBBM as a process driven by a backbone Brownian motion: $(h_t(X_t) :\: t \geq 0)$, representing the trajectory of the spine, from which (non-spine) particles branch out at Poissonian rate $2$ and then evolve independently according to the regular BBM law.

The connection between sBBM and regular BBM is given by the following ``Many-To-One'' lemma. This lemma is the ``workhorse'' of the spine method, being a useful tool in moment computations for the number of particles in a branching process satisfying a prescribed condition. Recall that for $t \geq 0$, we let $\cF_t$ denote the sigma-algebra generated by $(h_s, T_s, \rmd_s)_{s \leq t}$.
\begin{lem}[Many-To-One]
\label{l:4.1}
Let $F = (F(x) :\: x \in L_t)$ be a bounded $\cF_t$-measurable real-valued random function on $L_t$. Then, 
\begin{equation}
\bbE \Big( \sum_{x \in L_t} F(x) \Big)
= \rme^{t} \,\wt{\bbE} F(X_t) \,.
\end{equation}
\end{lem}

A direct application of the above lemma, together with the statistical structure of extreme values in the limit, gives the following expression for the law of the clusters. 
\begin{lem}[Lemma~5.1 in~\cite{CHL19}]
\label{l:7.0}
For all $r > 0$, there exists a probability measure $\nu_r$ on the space $\bbM_p((-\infty, 0])$ such that for all $u \in \bbR$, 
\begin{equation}
\label{e:387}
\wt{\bbP} \big( \cC_{t, r} (X_t) \in \cdot \, \big|\, \wh{h}_t(X_t) = \wh{h}_t^* = u\big) 
\underset{t \to \infty}\Longrightarrow \nu_r
\end{equation}
and
\begin{equation}
	\nu_r \underset{r \to \infty}\Longrightarrow \nu \,,
\end{equation}
where $\cC_{t,r}$ is as in~\eqref{e:5B} and $\nu$ is as in~\eqref{e:N7}.
Both limits hold in the sense of weak convergence on $\bbM((-\infty,0])$ equipped with the vague topology.
\end{lem}
We note that, in the proof in~\cite{CHL19}, the limits in $r$ and $t$ were taken at the same time, but the proof carries over to case where the limits are taken sequentially (in fact, the proof in~\cite{CHL19} relies on Theorem~2.3 in~\cite{ABBS2013}, in which the limit is taken sequentially).

\subsubsection{Decorated random walk representation}
\label{s:2.1.2}
Recall that $W=(W_s :\: s \geq 0)$ denotes a standard Brownian motion. For $0 \leq s \leq t$, set
\begin{equation}
\label{e:20.5}
\gamma_{t,s} := \tfrac{3}{2\sqrt{2}} \big(\log^+ s - \tfrac{s}{t}\log^+ t \big) 
\end{equation}
and let
\begin{equation}
\wh{W}_{t,s} := W_s - \gamma_{t,s} \,.
\end{equation}
We also let $H := \big(h^s := (h^s_t)_{t \geq 0} :\: s \geq 0\big)$ be a collection of independent copies of a regular BBM process $h$, that we will assume to be independent of $W$ as well. Finally, let $\cN$ be a Poisson point process with intensity $2 \rmd x$ on $\bbR_+$, independent of $H$ and $W$ and denote by $\sigma_1 < \sigma_2 < \dots$ its ordered atoms. The triplet 
\begin{equation}
\label{e:102.27}
(\wh{W}, \cN, H)	
\end{equation}
will be referred to as a {\em decorated random-walk-like process} (DRW). 

In what follows, we shall add a superscript $s$ to any notation used for objects which were previously defined in terms of $h$, to denote objects which are defined in the same way, but with $h^s$ in place of $h$. In this way, $\wh{h}^s_t$, $\wh{h}^{s*}_t$ and $\cE_t^s$ are defined as in the introduction only with respect to $h^s$ in place of $h$.
Using this notation, for $0 \leq r < R \leq t$ we define
\begin{equation}
\label{e:2.5}
\cA_t(r,R) := \Big \{
\max\limits_{k : \, \sigma_k \in [r,R]} \big(\wh{W}_{t,{\sigma_k}} + \wh{h}^{\sigma_k*}_{\sigma_k}\big) \leq 0 \Big\} \,,
\end{equation}
with the abbreivation
\begin{equation}
\label{e:102.7}
	\cA_t \equiv \cA_t(0,t)\,.
\end{equation}
This is the event that the Brownian motion $W$ plus a certain functional of its decorations $H$, stay below the curve $\gamma_{t,\cdot}$ at random times which correspond to the atoms of $\cN$. We shall refer to such event as a DRW {\em Ballot} event.

The connection between probabilities of the form~\eqref{e:387} and DRW Ballot probabilities are given by the following two lemmas which were proved in~\cite{CHL19}. 
\begin{lem}[Lemma~3.1 in~\cite{CHL19}]
\label{l:2.3}
For all $0 \leq r \leq t$ and $u, w \in \bbR$,
\begin{equation}
\label{e:29.2}
\wt{\bbP} \Big( \wh{h}_t^*\big(\rmB^\rmc_{r}(X_t)\big)
 \leq u \, \Big| \, \wh{h}_t(X_t) = w \Big) 
= \bbP \big( \cA_t(r,t) \, \big|\,
	\wh{W}_{t,r} = w - u,\, \wh{W}_{t,t} = -u \big) \,.
\end{equation}
In particular for all $t \geq 0$ and $v,w \in \bbR$,
\begin{equation}
\label{e:29.1}
\wt{\bbP} \big( \wh{h}^*_t \leq u \, \big| \, \wh{h}_t(X_t) = w \big) 
= \bbP \big( \cA_t \, \big|\, 
	\wh{W}_{t,0} = w - u,\, \wh{W}_{t,t} = -u \big) \,.
\end{equation}
\end{lem}

\begin{lem}[Lemma~3.2 in~\cite{CHL19}]
\label{l:5.2}
For all $0 \leq r \leq t$,
\begin{equation}
\label{e:26.6}
\wt{\bbP}  \Big( \cC_{t,r} (X_t) 
 \in \cdot \,\Big|\, \wh{h}^*_t = \wh{h}_t(X_t) = 0 \Big)  
= \bbP \bigg( \int_0^r \! \cE_{s}^{s} \big(\cdot - \wh{W}_{t,s} \big)\cN(\rmd s) \in \,  \cdot 
	 \ \Big| \ \wh{W}_{t,0} \! = \! \wh{W}_{t,t} \! = \! 0 \,;\, \cA_t \, \bigg).
\end{equation}
\end{lem}

The following lemma will be useful when estimating increments of $\wh{W}$.
\begin{lem}[Lemma~3.3 in~\cite{CHL19}] 
\label{l:lisa}
Let $s,r, t\in \bbR$ be such that $0 \leq r \leq r+s \leq t$, then
\begin{equation}\label{e:aser1}
 - 1  \, \leq \log^+ (r+s) - \left(\tfrac{t-(r+s)}{t-r} \log^+r + \tfrac{s}{t-r} \log^+t\right)
\leq \, 1+ \log^+ \big(s \wedge (t-r-s) \big).
\end{equation}
\end{lem}

\subsubsection{Ballot estimates for DRW}
In this subsection we include estimates for Ballot probabilities involving the DRW from the previous sub-section. They are either derived from~\cite{CHL17_supplement} in which more general DRW processes are treated, or taken directly from~\cite{CHL19}, in which the results from~\cite{CHL17_supplement} were specialized to the DRW from Subsection~\ref{s:2.1.2}. 

The first lemma provides bounds and asymptotics for the DRW Ballot event.
\begin{lem}[Essentially Lemma~3.4 in~\cite{CHL19}]
\label{lem:15}
There exists $C, C' > 0$ such that for all $0 \leq r < R\leq t$ and $w, v \in \bbR$,
\begin{equation}
\label{e:52}
\bbP \big( \cA_t(r,R)
	\,\big|\, \wh{W}_{t,r} = v \,,\,\, \wh{W}_{t,R} = w \big)
	\leq C \frac{(v^- + 1)(w^- + 1)}{R-r} \,.
\end{equation}
Also, there exists non-increasing functions $g: \bbR \to (0,\infty)$ and $f^{(r)}: \bbR \to (0,\infty)$ for $r \geq 0$, such that for all such $r$
\begin{equation}
\label{e:53}
\bbP \big( \cA_t(r,t)
	\,\big|\, \wh{W}_{t,r} = v \,,\,\, \wh{W}_{t,t} = w \big)  
	\sim 2 \frac{f^{(r)}(v) g(w)}{t-r},
\end{equation}
as $t \to \infty$ uniformly in $v,w$ satisfying $v,w < 1/\epsilon$ and $(v^-+1)(w^-+1) \leq t^{1-\epsilon}$ for any fixed $\epsilon > 0$. Moreover, 
\begin{equation}
\label{e:53.5}
\lim_{v \to \infty} \frac{f^{(r)}(-v)}{v} = \lim_{w \to \infty} \frac{g(-w)}{w} = 1  \,,
\end{equation}
for any $r \geq 0$. Finally there exists $f: \bbR \to (0, \infty)$ such that for all $v \in \bbR$, 
\begin{equation}
\label{e:54}
f^{(r)}(v) \overset{r \to \infty} \longrightarrow f(v) \,.
\end{equation}
\end{lem}
\begin{proof} This is nearly verbatim from Lemma~3.4 in~\cite{CHL19}, except that for the upper bounds, we allow the endpoint to be any $R\leq t$. The exact same proof there goes through (i.e.~applying Lemma \ref{l:lisa} with $R$ in place of $t$ to check the barrier curve condition of~\cite{CHL17_supplement} is satisfied), and we omit the repetitive details.
\end{proof}

The next lemma shows that under the Ballot event, the process $\wh{W}_{t,\cdot}$ is (entropically) repelled away from $0$.
\begin{lem}
\label{l:10.7.0}
For all $M > 0$, $2 \leq s \leq t/2$,
\begin{equation}
\bbP \Big( \max_{u\in[s,\, t/2]}
	\wh{W}_{t,u} > -M \,\Big|\, \wh{W}_{t,0} = 0 \,,\,\, \wh{W}_{t,t} = 0\,,\,\,\cA_t\Big)
	\leq C \frac{(M+1)^2}{\sqrt{s}} \,.
\end{equation}
\end{lem}

\begin{proof}
Lemma~2.7 in~\cite{CHL17_supplement}, specialized to $W, (-\wh{h}^{s*}_s)_{s \geq 0}, \cN, \gamma$ as $W, Y, \cN, \gamma$ immediately gives
\begin{equation}
\bbP \Big( \max_{s\in[s,\, t/2]} \wh{W}_{t,u} > -M,\, \cA_t \,\Big|\, \wh{W}_{t,0} = 0 \,,\,\, \wh{W}_{t,t} = 0\Big)
	\leq C \frac{(M+1)^2}{t\sqrt{s}} \,,
\end{equation}
after noting that $\cQ_t(0,t)$ identifies with $\cA_t$ under the above specialization. Together with the lower bound,
\begin{equation}
\bbP \big( \cA_t \,\big|\, \wh{W}_{t,0} = 0 \,,\,\, \wh{W}_{t,t} = 0\big)
	\geq c \frac{1}{t} \,,
\end{equation}
which is implied by Lemma~\ref{lem:15}, this gives the desired statement.
\end{proof}

The next lemma shows that one can replace the restriction on $\wh{W}_{t,u}$ for $u \in [s,t-s]$ under $\cA_t$ with the condition that $W_{u}$ (without the additional $\gamma_{t,s}$) stays negative within this interval, with the difference between the probabilities of the two events being of smaller order.
\begin{lem}
\label{l:2.13}
For $0 < r < R < t$, denote the event
\begin{equation} \wt{\cA}_{t}(r, R) := \{\cA_t(0,r) \cap \cA_t(R,t) \cap \{\max W_{[r,R]} \leq 0\} \}.
\end{equation}
There exists $C > 0$ such that for all $2 < s < t/2$ \,,
\begin{equation}
    \bbP \Big(\cA_{t} \, \Delta \,  \wt{\cA}_{t}(s, t-s)\,\Big|\, \wh{W}_{t,0} = 0 \,,\,\, \wh{W}_{t,t} = 0 \Big) \leq C s^{-1/4} t^{-1} \,.
\end{equation}
\begin{proof}
The proof of this exact statement appears in the beginning of the proof of Lemma~2.8 in~\cite{CHL17_supplement}, specialized to $W, (-\wh{h}^{s*}_s)_{s \geq 0}, \cN, \gamma$ as $W, Y, \cN, \gamma$.
\end{proof}

\end{lem}

\section{Almost sure growth of extremal level sets}
\label{s:3}
    In this section we prove Theorem~\ref{1st_result}. 
In what follows, we work with a probability space on which $\cE^*$, $\cE$ and $(\cC^k)_k$ are defined as in~\eqref{e:5.5} and~\eqref{e:O.2}. Following~\cite{CHL19}, given a subset $B \subseteq \bbR$ we define
\begin{equation}
\cE (\cdot \;; B) := \sum_{k} \cC^k ( \cdot - u_k) 1_{\{u_k \in B\}}  \,,
\end{equation}
This process is a {\em restricted} version of the limiting extremal process $\cE$, in which only contributions form clusters whose tip is at $B$ are recorded.

As mentioned in Subsection~\ref{s:1.1}, for $v > 0$, we control the contributions from tips above and below $-\delta \log v$ separately, where $\delta$ is any real number in $(0, 1/\sqrt{2})$. For the first of the two, the main tool is the following moment estimate, which was essentially derived in~\cite{CHL19}.
\begin{lem}
\label{l:3.1}
 Let $-\infty < -v \leq w \leq z < \infty$. Then
	\begin{equation}\label{eq:cluster_1st_moment_estimate}
		\bbE \left[\cE\big([-v, \infty) ; \; [w, z] \big)\big\vert Z\right] \sim C_\star Z e^{\sqrt{2}v} (z-w)
	\end{equation}
as $w+v \to \infty$, uniformly in $z$, and as $z+v\to \infty$, uniformly in $w$. Moreover, there exists constants $C \in (0, \infty)$ such that
	\begin{align}
		\bbE \left[\cE\big([-v, \infty) ; \; [w, z] \big)\big\vert Z\right] &\leq C Z e^{\sqrt{2}v} (z-w) \label{eq:cluster_1nd_moment_estimate_upper} \\
		\Var \left[\cE\big([-v, \infty) ; \; [w, z] \big)\big\vert Z\right] &\leq C Z e^{2\sqrt{2}v+\sqrt{2}z} (z+v). \label{eq:cluster_2nd_moment_estimate}
	\end{align}
\end{lem}

\begin{proof} The upper bounds \eqref{eq:cluster_1nd_moment_estimate_upper} and \eqref{eq:cluster_2nd_moment_estimate} are equation (6.5) and (6.6) in \cite{CHL19}. The asymptotic \eqref{eq:cluster_1st_moment_estimate} is equation (6.7) in \cite{CHL19}, though only stated there for the regime $w+v\to \infty$, uniformly in $z$. Let us check \eqref{eq:cluster_1st_moment_estimate} for the regime $z+v\to \infty$, uniformly in $w$. Take $w'$ with $w'+v = \sqrt{z+v}$, then $w'+v \to \infty$ and $w'+v = o(z+v)$ as $z+v\to \infty$. Then we have
	\begin{align}
		\bbE \left[\cE\big([-v, \infty) ; \; [w, z] \big)\big\vert Z\right]
		&= \bbE \left[\cE\big([-v, \infty) ; \; [w, w \vee w'] \big)\big\vert Z\right] + \bbE \left[\cE\big([-v, \infty) ; \; [w \vee w', z] \big)\big\vert Z\right] \\
		&= O(1) Z e^{\sqrt{2}v} (w'-w)^+ + (1+o(1)) C_\star Z e^{\sqrt{2}v} (z-w - (w'-w)^+)
	\end{align} 
with $O(1), o(1)$ are uniform in the regime considered, where we used \eqref{eq:cluster_1st_moment_estimate} and \eqref{eq:cluster_1nd_moment_estimate_upper}. For $w \geq w'$ the conclusion is then immediate and for $w < w'$ it follows from the observation that $w'-w\leq w'+v = o((z+v)-(w'+v)) = o((z-w)-(w'-w)).$
	
\end{proof}

For the contribution coming from tips above $-\delta \log v$, we shall need 
the following key proposition, which shows that it is almost surely of negligible order compared to $\cE([-v, \infty))$.
\begin{prop}
\label{higher_leaders} For any $0< \delta < \frac{1}{\sqrt{2}}$ and $K \geq 0$, it holds almost surely that
	\begin{equation}\label{higher_leaders_est}
		\lim_{v\rightarrow \infty} \frac{\cE\big([-v, \infty) ; \; [-\delta \log v, \infty) \big)}{v^{-K} \rme^{\sqrt{2} v}} = 0.
	\end{equation}
\end{prop}
	
The proof of Proposition~\ref{higher_leaders}, which constitutes the main effort in the overall argument, will be relegated to the next subsection, in favor of first showing how the above two results can be used the prove Theorem~\ref{1st_result}.

\begin{proof}[Proof of Theorem~\ref{1st_result}] 
Fix $0 < \delta < \frac{1}{\sqrt{2}}$. By Proposition \ref{higher_leaders}, almost surely
	\[
	\lim_{v\rightarrow \infty}  \frac{\cE([-v, \infty) ; \; [-\delta \log v, \infty))}{v e^{\sqrt{2} v}} = 0.
	\]
It remains to show that almost surely
\begin{equation}\label{eq:proof_main_main}
\lim_{v\rightarrow \infty} \frac{\cE\big([-v, \infty) ; \; [-v, -\delta \log v] \big)}{v e^{\sqrt{2} v}} = C_\star Z.
\end{equation}
Our strategy is to show the limit holds along sequence $v \in \eps \mathbb{N}$ for any ``mesh size'' $\eps \in (0, 1)$ and then take $\eps \rightarrow 0$. Precisely, for $n\geq 1/\eps$, we denote
\begin{align}
	F_\eps(n) &:= \frac{\cE\big([-n\eps, \infty) ; \; [-n\eps, -\delta \log ((n+1)\eps))\big)}{(n+1)\eps e^{\sqrt{2} \, (n+1)\eps}} \,,\\
	G_\eps(n) &:= \frac{\cE\big([-(n+1)\eps, \infty) ; \; [-(n+1)\eps, -\delta \log (n\eps))\big)}{n\eps e^{\sqrt{2} \, n\eps}} \,.
\end{align}
Then for $v \in [n\eps, (n+1)\eps)$, 
\begin{equation}\label{eq:proof_main_squeeze}
F_\eps(n) \leq \frac{\cE\big([-v, \infty) ; \; [-v, -\delta \log v] \big)}{v e^{\sqrt{2} v}} \leq G_\eps(n).
\end{equation}
By \eqref{eq:cluster_1st_moment_estimate} and \eqref{eq:cluster_2nd_moment_estimate},
	\begin{align}
	\bbE \left[F_\eps(n)\big\vert Z\right] &\sim \frac{C_\star Z n\eps e^{\sqrt{2}\, n\eps}}{(n+1)\eps e^{\sqrt{2} \, (n+1)\eps}} \sim C_\star Z e^{-\sqrt{2} \eps}, \\
	\Var \left[F_\eps(n)\big\vert Z\right] &\leq \frac{C Z (n\eps)^{1 - \delta \sqrt{2} } e^{2\sqrt{2} \, n\eps}}{[(n+1)\eps]^2 e^{2\sqrt{2} \, (n+1)\eps}} \lesssim Z (n\eps)^{-1 - \delta \sqrt{2} } \,,
\end{align}
so that by Chebyshev's inequality, this implies that for any $\alpha > 0$,
\begin{align}\label{eq:proof_main_Chebyshev}
	\bbP \left[\left| F_\eps(n) - C_\star Z e^{-\sqrt{2} \eps} \right| > \alpha \bigg\vert Z\right] \lesssim \alpha^{-2} Z (n\eps)^{-1 - \delta \sqrt{2} } (1 + o(1))
\end{align}
where the rate of $o(1)$ depends on $\alpha, \eps, Z$. Because $\delta > 0$, the right side of \eqref{eq:proof_main_Chebyshev} is summable over $n \in \mathbb{N}$. A standard application of Borel-Cantelli then gives the almost surely,
\begin{equation}\label{eq:proof_main_B-C}
	\lim_{n\rightarrow \infty} F_\eps(n) = C_\star Z e^{-\sqrt{2} \eps}. 
\end{equation}
Similarly, repeating the same argument for $G_\eps(n)$, we also have that almost surely
\begin{equation}\label{eq:proof_main_B-C_2}
	\lim_{n\rightarrow \infty} G_\eps(n) = C_\star Z e^{\sqrt{2} \eps}. 
\end{equation}
Combining the last two a.s.~limits with the condition \eqref{eq:proof_main_squeeze}, we obtain the a.s.~limit \eqref{eq:proof_main_main} after taking $\eps\rightarrow 0$.
\end{proof}

\subsection{Typical contribution from high clusters is negligible}

Proposition \ref{higher_leaders} will follow from the following lemma which bounds the typical asymptotic growth of the cluster level sets.

\begin{lem}
\label{key_lemma} For every $\eps > 0, K \geq 0$, 
	\begin{equation}\label{key_lemma_est}
		\bbP\Big(\exists w \geq u: \mathcal{C}\big([-w, 0]\big) \geq \eps w^{-K} \rme^{\sqrt{2} w} \Big)  \lesssim_{\eps, K} \frac{1}{u} (\log u)^{5/2}
	\end{equation}
	for all $u \geq 2$.
\end{lem}

Before proving this lemma, let us state two corollaries of it, and then use both to prove Proposition~\ref{higher_leaders}. The first corollary is immediate.
\begin{cor}\label{key_lem_cor_1} Let $K \geq 0$. Then, almost surely for all $v$ large enough,
	\[
	\mathcal{C}\big([-v, 0]\big) < v^{-(K+1)} \rme^{\sqrt{2} v} \,.
	\]
\end{cor}

For the second, recall that $(\cC^k)_k$ are i.i.d. point measures, each drawn from the law $\nu$ and that $(u_k)_{k \geq 1}$ enumerate the atoms of $\cE^*$ in decreasing order. 
\begin{cor}\label{key_lem_cor} Given $0< \delta < 1/\sqrt{2}$ and $K \geq 0$, define $\varphi(u) := \rme^{u/\delta}$ and the associated event
	\begin{equation}\label{eq:bad_leader_event}
		\cD_k := \big\{\exists w \geq \varphi(-u_k): \mathcal{C}^{k}\big([-w, 0]\big) \geq w^{-(K+1)} \rme^{\sqrt{2} w} \big\}.
	\end{equation}
	Then the set 
	\begin{equation}\label{eq:set_bad_leaders}
		\cU^*_{\text{bad}} := \big\{k \geq 0 :\: u_k \leq 0,\, \cD_k \ \text{happens} \big\}
	\end{equation}
	is almost surely finite.  
	\end{cor}
	\begin{proof}[Proof of Corollary \ref{key_lem_cor}] 
Denote by $\cE^*_-$ the restriction of $\cE^*$ to $(-\infty, 0)$. By Lemma \ref{key_lemma} (with $\eps = 1$),
		\begin{equation}
		\begin{split}
		\bbE \big(\big|\cU^*_{\text{bad}}\big|\,\big|\, \cE^*\big) & = 
			\sum_{u_k \in \cE^*_-} \bbP \big(\cD_k\,\big|\,\cE^*\big) \lesssim_K \sum_{u_k \in\cE^*_-} \varphi(-u_k)^{-1} [1 + \log \varphi(u_k)]^{5/2} \\ & \lesssim_{\delta, K} \sum_{u_k \in\cE^*_-} e^{-u_k^-/\delta} (1+u_k^-)^{5/2}. \label{eq:proof_higher_leaders_B-C}
			\end{split}
		\end{equation}
		Therefore in expectation, conditional on $Z$ we have
		\begin{align}
		\bbE \big(\big|\cU^*_{\text{bad}}\big|\,\big|\,Z\big)
		&\lesssim_{\delta, K}  \int_{-\infty}^0 e^{-u^-/\delta} (1+u^-)^{5/2}\, Z e^{-\sqrt{2} u} \rmd u < \infty \,,
		\end{align}
		almost surely, where we have used that $\delta < 1/\sqrt{2}$. This implies that $\cU^*_{\text{bad}}$ is almost surely finite under $\bbP(-|Z)$ and hence that also under $\bbP$.
	\end{proof}
	
We can now give
\begin{proof}[Proof of Proposition \ref{higher_leaders}] Recall that we have
\begin{equation}
		\cE\big([-v, \infty) ; \; [-\delta \log v, \infty)\big) 
		= \sum_{u_k \geq -\delta \log v} \cC^k \big([-(v+u_k), 0]\big) \,.
 \label{higher_leaders_total}
\end{equation}
Denote by $\cE^*_+$ the restriction of $\cE^*$ to $[0, \infty)$. The sum in \eqref{higher_leaders_total} is bounded by
	\begin{align}
	\cE\big([-v, \infty) ; \; [-\delta \log v, \infty)\big) \leq  \sum_{\substack{u_k \in [-\delta \log v, 0] \\ k \notin \cU^*_{\text{bad}}}} \cC^{k} \big([-v, 0]\big) + \sum_{\substack{u_k \in \cE^*_+ \\
		\text{or } k \in \cU^*_{\text{bad}}}} \cC^{k} \big([-(v+u_k), 0]\big) \,,
	\label{higher_leaders_total_1}
\end{align}
where $\cU^*_{\text{bad}}$ is the set defined in \eqref{eq:set_bad_leaders}. For the first sum above, since $u_k \geq -\delta \log v$, we have $\varphi(-u_k) \leq v$, so that $k \notin \cU^*_{\text{bad}}$ implies 
\begin{equation}
	\mathcal{C}^{k}\big([-v, 0]\big) < v^{-(K+1)} \rme^{\sqrt{2} v}. 
\end{equation}
Thus, the first sum in \eqref{higher_leaders_total_1} is bounded by 
\begin{align}
	\cE^*([-\delta \log v, 0]) \times v^{-(K+1)} \rme^{\sqrt{2} v} \,,
\end{align}
which, by~\eqref{e:1.3} is asymptotically
\begin{align}
	 (1+o(1)) \frac{Z}{\sqrt{2}} v^{\sqrt{2} \delta} v^{-(K+1)} \rme^{\sqrt{2} v} = o(1) Z v^{-K} \rme^{\sqrt{2} v} \label{higher_leaders_negative_2}.
\end{align}
with $o(1) \to 0$ as $v \to \infty$ almost surely.

Turning to the second sum in~\eqref{higher_leaders_total_1}, by Corollary \ref{key_lem_cor_1}, the $k$-th term in the sum is at most
\begin{align}
\label{e:103.24}
(1+o(1)) v^{-(K+1)} e^{\sqrt{2}(v+u_k)} = o(1)\, v^{-K} e^{\sqrt{2} v} \,,
\end{align}
where $o(1) \to 0$ as $v \to \infty$ almost surely. Moreover, the 
set $\cU^*_{\text{bad}}$ is finite a.s.~by Corollary \ref{key_lem_cor}, and the set $\cE^*_+$ is finite a.s.~because the intensity $Z e^{-\sqrt{2}} dx$ is a.s. integrable on  $[0, \infty)$. Since these sets also do not depend on $v$, the right hand side of~\eqref{e:103.24} is also a bound on the entire second sum in~\eqref{higher_leaders_total_1}. Combining both bounds we recover the statement of the proposition.
\end{proof}

\subsection{Upper bound on typical growth of cluster level sets}
It remains to prove Lemma~\ref{key_lemma}. Here we appeal to the weak cluster law representation in Subsection~\ref{s:2.1}. The key idea here, and the reason for the bound in the lemma to hold, is that the (time-reversed) backbone spine Brownian motion $(h_{t-s}(X_{t-s}))_{s=0}^t$ is naturally entropically repelled under the conditioning on $\{\wh{h}_t(X_t) = \wh{h}_t^* = u\}$, and thus the BBMs that branch off it at times $t-O(1)$ contribute only $o(\rme^{\sqrt{2}v})$ to level set $[-v, 0]$. 

To make a precise statement, for $u \in \bbR$ and $t \geq 0$, we define the repulsion event,
\begin{equation}
	\cB_{u}(t) :=  \Big\{ \max_{s\in[\eta_u u^2,\, t/2]} 
	\big(h_{t-s}(X_{t-s}) - m_t + m_s\big) > -M_u \Big\}
\end{equation}
for some positive parameters $\{\eta_u\}$ and $\{M_u\}$, to be specified later. We shall then show
\begin{lem}\label{trunc_1st_moment} For all $u \geq 1$, 
	\begin{equation}\label{atypical_event_est}
		\wt{\bbP} \Big(\cB_{u}(t) \Big|\, \wh{h}^*_t = \wh{h}_t(X_t) = 0\Big) \lesssim \frac{1}{u}\, (M_u + 1)^2\eta_u^{-1/2},
	\end{equation}
and there exists constant $C'$ such that for $w\geq u$,
	\begin{equation}
		\label{1st_moment_est}
		\wt{\bbE}\Big(\cC_{t,r_t}^*\big([-w,0]\big) ;\; \cB_{u}(t)^{\rmc} \, \Big| \,  \wh{h}^*_t =  \wh{h}_t(X_t) = 0\Big)  \lesssim \rme^{\sqrt{2}w} \Big(w \rme^{-w/2} +  \rme^{-\frac{1}{17\eta_u}} + \rme^{-C' M_u} + w t^{-1/2}\Big).
	\end{equation}
\end{lem}
The second statement, albeit with a slightly coarser upper bound, was proved in~\cite{CHL20} (see Lemma~3.2). With this lemma at hand we can give
\begin{proof}[Proof of Lemma~\ref{key_lemma}]
Observe that it is enough to show that for all $u \geq 0$,
\begin{equation}
	\bbP\Big(\exists w \in [u, 2u]: \mathcal{C}\big([-w, 0]\big) \geq \eps w^{-K} \rme^{\sqrt{2} w} \Big)  \lesssim_{\eps, K} \frac{1}{u} (\log u)^{5/2},
\end{equation}
because summing this estimate over the geometric sequence $u, 2u, 4u, \cdots$ yields \eqref{key_lemma_est} (with a different implicit constant factor in front). In turn, by Lemma~\ref{l:7.0} we just need to show the corresponding estimate for $\cC_{t,r_t}^*$, namely that
\begin{equation}
	\label{key_lemma_prelimit_est}
	\wt{\bbP}\Big(\exists w \in [u, 2u]: \cC_{t,r_t}^*\big([-w,0]\big) \geq \eps w^{-K} \rme^{\sqrt{2}w} \, \Big| \,  \wh{h}^*_t =  \wh{h}_t(X_t) = 0\Big)  \lesssim_{\eps, K} \frac{1}{u} (\log u)^{5/2}
\end{equation}
holds for all $t$ large enough and fixed $u$.

Now, if we choose $M_u, 1/\eta_u = \Theta(K \log u)$, then the right hand side of \eqref{atypical_event_est} identifies with the right hand side of \eqref{key_lemma_prelimit_est}. It therefore remains to show that the probability of the event in \eqref{key_lemma_prelimit_est} on $\cB_{u}(t)^{\rmc}$ is of lower order. This will come from the truncated first moment estimate \eqref{1st_moment_est}. Indeed, by Markov's inequality, we have for any integer $w \in [u, 2u+1]$,
\begin{multline}
	\label{1st_moment_est_2}
	\wt{\bbP}\Big(\cC_{t,r_t}^*\big([-w,0]\big) \geq \eps w^{-K} \rme^{\sqrt{2} w}  ;\; \cB_u(t)^{\rmc} \, \Big| \,  \wh{h}^*_t =  \wh{h}_t(X_t) = 0\Big) \\ \lesssim \, \eps^{-1} w^K \Big(w \rme^{-w/2} +  \rme^{-\frac{1}{17\eta_u}} + \rme^{-C' M_u} + w t^{-1/2}\Big) \lesssim \eps^{-1} w^{-(K+3)}
\end{multline}
for appropriate choice of $M_u, 1/\eta_u = \Theta(K \log u)$, holding for all $t \geq u^{4K+8}$. This can be now boosted to a ``maximal inequality'' over $w\in [u, 2u]$ using monotonicity and Union Bound:
\begin{equation}
\begin{split}
	\label{1st_moment_est_3}
	\wt{\bbP}\Big(\exists w \in [u,& 2u]: \cC_{t,r_t}^*\big([-w,0]\big) \geq \eps w^{-K} \rme^{\sqrt{2} w}  ;\; \cB_u(t)^{\rmc} \, \Big| \,  \wh{h}^*_t =  \wh{h}_t(X_t) = 0\Big) \\
	\leq \, & \wt{\bbP}\Big(\exists \ w \in [u, 2u+1]  \cap \bbZ: \cC_{t,r_t}^*\big([-w,0]\big) \geq \eps w^{-K} \rme^{\sqrt{2} (w-1)}  ;\; \cB_u(t)^{\rmc} \, \Big| \,  \wh{h}^*_t =  \wh{h}_t(X_t) = 0\Big) \\
	\leq \, & \sum_{w\in [u, 2u+1]  \cap \bbZ } \wt{\bbP}\Big(\cC_{t,r_t}^*\big([-w,0]\big) \geq \rme^{-\sqrt{2}} \eps w^{-K} \rme^{\sqrt{2} w}  ;\; \cB_u(t)^{\rmc} \, \Big| \,  \wh{h}^*_t =  \wh{h}_t(X_t) = 0\Big) \\
	\lesssim \, & \sum_{w\in [u, 2u+1] \cap \bbZ } \rme^{\sqrt{2}} \eps^{-1} w^{-(K+3)} \lesssim \eps^{-1} u^{-(K+2)}.
\end{split}
\end{equation}
Combining \eqref{1st_moment_est_2} and \eqref{1st_moment_est_3} together, we conclude that for all $t \geq u^{4K+8}$,
\begin{multline}
\wt{\bbP}\Big(\exists w \in [u, 2u]: \cC_{t,r_t}^*\big([-w,0]\big) \geq \eps w^{-K} \rme^{\sqrt{2} w}  \Big| \,  \wh{h}^*_t =  \wh{h}_t(X_t) = 0\Big) \\
 \lesssim \,  K^{5/2} \, \frac{1}{u} (\log u)^{5/2} + \eps^{-1} u^{-(K+2)},
\end{multline}
which yields \eqref{key_lemma_prelimit_est} by noticing that the second term is of lower order. 
\end{proof}

Finally, it remains to give
\begin{proof}[Proof of Lemma \ref{trunc_1st_moment}] 
The first part of the lemma follows by identifying the left hand side in~\eqref{atypical_event_est} with the right hand side of the probability in the statement of Lemma~\ref{l:10.7.0} with $M = M_u$ and $s = \eta_u u^2$, via Lemma~\ref{l:5.2}. The result then follows by the statement of Lemma~\ref{l:10.7.0}. For the second statement, it turns out that one just needs a more careful estimation of the bound obtained in the proof of Lemma~3.1 of~\cite{CHL20}. Indeed, the computation there shows that
\begin{equation}
	\wt{\bbE}\big(\cC_{t,r_t}^*\big([-w,0]\big) ;\; \wh{h}^*_t \leq 0 ,\,\cB_{u}(t)^{\rmc} \, \big| \,  \wh{h}_t(X_t) = 0\big)  
\end{equation}
is upper bounded by 
\begin{equation}
\label{e:3.34}
\frac{C}{t} \rme^{\sqrt{2} w} (w+1) \int_{s=0}^\infty \frac{\rme^{-w^2/(16s)}  + \rme^{-w/2}}{ \sqrt{s} (s+1)}  
	\Big(1_{ \{ s \in [\eta_u u^2, \, t/2]^\rmc \}} + \rme^{-C' M_u} 1_{\{ s \in [ \eta_u u^2, \, t/2] \}} \Big) \rmd s\,.
\end{equation}
Proceeding by a more careful estimation, the last integral is at most
\begin{equation}
\begin{split}	
	\int_{s=0}^\infty & \frac{\rme^{-w/2}}{\sqrt{s} (s+1)}  \rmd s + \int_{s=0}^{\eta_u u^2} \frac{\rme^{-w^2/(16s)}}{s^{3/2}}  \rmd s + \rme^{-C' M_u} \int_{s=0}^{\infty} \frac{\rme^{-w^2/(16s)}}{s^{3/2}}  \rmd s + \int_{s=t/2}^{\infty} \frac{1}{s^{3/2} }  \rmd s \\
	\lesssim \, & \rme^{-w/2} + \frac{1}{w} \int_{s=0}^{\eta_u} \frac{\rme^{-1/(16s)}}{s^{3/2}}  \rmd s + \rme^{-C' M_u} \frac{1}{w} \int_{s=0}^{\infty} \frac{\rme^{-1/(16s)}}{s^{3/2}}  \rmd s + t^{-1/2} \\
	\lesssim \, & \rme^{-w/2} + \frac{1}{w}\, \rme^{-1/(17\eta_u)} + \frac{1}{w}\, \rme^{-C' M_u} + t^{-1/2}. 
\end{split}
\end{equation}
Plugging this in~\eqref{e:3.34} and dividing by
\begin{equation}
		\wt{\bbP}\big(\wh{h}^*_t \leq 0\, \big| \,  \wh{h}_t(X_t) = 0\big)  
		= \bbP\big(\cA_t\,\big | \,\wh{W}_{t,0} = \wh{W}_{t,t} = 0\big) \gtrsim \frac{1}{t}\,,
\end{equation}
which is a consequence of Lemma~\ref{l:2.3} and Lemma~\ref{lem:15}, the result follows.
\end{proof}

\section{Cluster size fluctuations}
\label{s:5}
In this section we prove Theorem~\ref{2nd_result}. The proof relies on the strong representation of the cluster law as given by Theorem~\ref{p:1}. Throughout this section, we shall work on an underlying probability space, with measure $\bbP$ and expectation $\bbE$, on which the processes $\wc{W}$, $\wc{\cN}$, $\wc{\cE}$ and $\wc{Y}$ are defined, and satisfy:
\begin{enumerate}[label=(\alph*)]
	\item\label{itm:a} $(\wc{W}$, $\wc{\cN}$, $\wc{\cE})$ is a process whose existence is given by Theorem~\ref{p:1}.
	\item\label{itm:b}  $\wc{Y} =  (\wc{Y}_s)_{s \geq 0}$ satisfies 
	\begin{equation}
	\label{e:105.1}
	\wc{Y}_s = -Y_s - \frac{3}{2\sqrt{2}} \log^+ s \,,	
	\end{equation}
 	where $Y := (Y_s)_{s \geq 0}$ is a Bessel-3 process starting from $y_0 = 0$.
	\item\label{itm:c} $\wc{Y}$ and $\wc{\cE}$ are independent.
\end{enumerate}
Given any $v > 0$, we introduce the scaled time variable $\hat{s}$ given by
\begin{equation}
\hat{s} := v^{-\frac{4}{3}} s \,,
\end{equation}
the rescaled version $Y^{(v)}$ of $Y$, 
		\begin{equation}
			Y^{(v)}_{\hat s} := v^{-\frac{2}{3}} Y_{s} 
			\ ; \quad s \geq 0 \,, \label{Y_rescaled_def}
		\end{equation}
and a version of $\zeta$ from Theorem~\ref{2nd_result}, which is  defined w.r.t. $Y^{(v)}$,
		\begin{equation}
			\zeta^{(v)} := \inf_{\hat{s} > 0} \Big(\sqrt{2}\, Y^{(v)}_{\hat{s}} + \frac{1}{2\hat{s}}\Big) \,. \label{chi_rescaled_def}
		\end{equation}
Observe that in view of the scaling invariance of Bessel-3 \eqref{e:202.4}, 
\begin{equation}
\label{e:205.3}
	(Y^{(v)}, \zeta^{(v)}) \overset{\rmd}= (Y, \zeta)\,.
\end{equation}

Next, on the above probability space, we introduce the following event.
\begin{definition} 
For $K, \Lambda > 0$, $\epsilon \in (0,1)$ and $v > 0$, let $\cV^{(v)}(\eps, K, \Lambda)$ be the event that the following occur:
	\begin{enumerate}[label=(\roman*)]
		\item\label{itm:1} For all $s \geq K$,
		\begin{equation}\label{eq_Bessel}
			\wc{W}_s = \wc{Y}_s \,.
		\end{equation}
		\item\label{itm:2} For all $s \geq K$,
		\begin{equation}\label{eq_D-E} 
			s^{\frac{1}{2}-\eps} \leq Y_s \leq 3(s \log \log s)^{\frac{1}{2}} \,.
		\end{equation}
		\item\label{itm:2.5} For all $s < K$,
		\begin{equation}\label{eq_UB} 
			\wc{W}_s \leq \Lambda \,.
		\end{equation}
		\item\label{itm:3} The process $Y^{(v)}$
		is $(\frac{1}{2}-\eps)$-H\"older for $\hat s \in [0, K]$ with H\"older constant $\Lambda$.
		\item\label{itm:4} For all $\hat{s} \geq K$,
		\begin{equation}\label{eq_D-E_hat} 
		\hat{s}^{\frac{1}{2}-\eps} \leq Y^{(v)}_{\hat{s}} \leq 3( \hat{s} \log \log \hat{s})^{\frac{1}{2}}.
		\end{equation}
		\item\label{itm:5} 
		With $\delta_\eps := c \log(1/\eps)^{-1/2}$ for some absolute constants $c > 0$
		\begin{equation}
			\zeta^{(v)} = \min_{\hat{s} \in [\delta_\eps, K]} \Big(\sqrt{2}\, Y^{(v)}_{\hat{s}} + \frac{1}{2\hat{s}}\Big) \leq (2\delta_\eps)^{-1} \,.\label{range_chi_v}
		\end{equation} 
		
		\item\label{itm:6} For every $\hat s \in [0, K]$, there exists $\sigma \in \wc{\cN}$ whose rescaled version $\hat{\sigma} := v^{-\frac{4}{3}} \sigma$ satisfies 
		\begin{equation}
			|\hat{\sigma} - \hat{s}| \leq K v^{-\frac{4}{3} + \eps}.
		\end{equation}
		\item\label{itm:7} For all $s \geq K$, it holds that $\wc{\cN}([0, s]) \leq 3s$.
	\end{enumerate}
\end{definition}

Theorem~\ref{2nd_result} will follow immediately from the cluster representation given by Theorem~\ref{p:1} together with the next two key propositions. The first shows that we may define the probability space and processes above such that the event $\cV^{(v)}(\eps, K, \Lambda)$ holds with high probability.
\begin{prop}\label{p:nice_event_whp} For any $\eps > 0$, there exists some $K, \Lambda < \infty$ such that for all $v \geq 1$, we may define the processes $(\wc{W}$, $\wc{\cN}$, $\wc{\cE})$ and $\wc{Y}$ on the same probability space, such that~\ref{itm:a},~\ref{itm:b} and~\ref{itm:c} hold, and 
\begin{equation}
	\bbP \big(\cV^{(v)}(\eps, K, \Lambda) \big) \geq 1 - \eps \,.
	\end{equation}
\end{prop}
The second shows that, conditional on $\cV^{(v)}(\eps, K, \Lambda)$ the normalized deviation of the size of the cluster level-set $[-v,0]$, tends in probability to $\zeta^{(v)}$.
\begin{prop}\label{2nd_result_quantitative} For every $\eps > 0$ small enough and $K, \Lambda < \infty$ large enough, there exists $v_0 = v_0(\eps, K, \Lambda) < \infty$ such that for every $v \geq v_0$, on the event $\cV^{(v)}(\eps, K, \Lambda)$,
	\begin{equation}
		\bbP \bigg(\bigg|\frac{\log \wc{\cC}([-v, 0]) -\sqrt{2} v }{- v^{\frac{2}{3}}} - \zeta^{(v)}\bigg| > \epsilon \,\bigg|\, \wc{W}, \wc{\cN}, \wc{Y}  \bigg) <   \epsilon \,.
	\end{equation}
\end{prop}

Assuming the above two, we can readily give
\begin{proof}[Proof of Theorem~\ref{2nd_result}]
Combining Proposition~\ref{p:nice_event_whp} and Proposition~\ref{2nd_result_quantitative},  we get
\begin{equation}
\label{e:5.11}
	\bbP\left(\left|\frac{\log \wc{\cC}([-v, 0]) -\sqrt{2} v }{- v^{\frac{2}{3}}} - \zeta^{(v)}\right| > \eps \right) \to 0 
\end{equation}
as $v \to \infty$ for any given $\eps > 0$. Since $\zeta^{(v)}$ and $\wc{\cC}$ have the same laws, respectively, as $\zeta$ and $\cC$, this completes the proof.
\end{proof}

\subsection{Proof of key propositions}
\begin{proof}[Proof of Proposition~\ref{p:nice_event_whp}]
Existence of a probability space which carries the process $(\wc{W}$, $\wc{\cN}$, $\wc{\cE})$ is given by Theorem~\ref{p:1}. The fact that we can define the process $\wc{Y}$ on the same (possibly larger version of the original) space such that~\eqref{eq_Bessel} holds with arbitrarily (but given) high probability provided $K$ is large enough is a standard consequence of Part~2 of Theorem~\ref{p:1} and the existence of a coupling under which the Total Variation distance identifies with the probability of distinguishability. The independence between $(\wc{W},\wc{\cN})$ and $\wc{\cE}$ implies that we can define $\wc{Y}$ such that it is independent of $\wc{\cE}$ as well.

 Arbitrarily high probability, provided $K$ and then $\Lambda$ are chosen large enough, is guaranteed for the probability of Conditions~\ref{itm:2},~\ref{itm:3},~\ref{itm:4},~\ref{itm:6}, uniformly in $v$,  thanks to Lemma~\ref{uniform_continuity}, Lemma~\ref{D-E} and Theorem~\ref{p:1} and the law invariance~\eqref{e:205.3}. For any $K$, Condition~\ref{itm:2.5} holds for large enough $\Lambda$ thanks to the continuity of $\wc{W}$ on the compact interval $[0,K]$.
 For~\ref{itm:5}, in view of~\eqref{e:205.3}, we may prove~\eqref{range_chi_v} for $\zeta$ in place of $\zeta^{(v)}$. We then observe that $\zeta$ necessarily has sub-Gaussian right tails since
	 \begin{equation}
	 \bbP(\zeta > x) \leq \bbP\big(\sqrt{2} Y_1 + \tfrac{1}{2} > x\big) \leq C e^{-cx^2} \,,
	 \end{equation}
	for some absolute constants $c, C > 0$. Together with the growth behavior given by condition~\ref{itm:4} and positivity of $Y$, this gives~\eqref{range_chi_v}.

	Lastly for \ref{itm:6}, by the third assertion in Theorem~\ref{p:1}, with probability at least $1 - Ce^{-c K v^\eps}$ there exists $\sigma \in \wc{\cN}$ such that
	\begin{align}
		|\sigma - s| \leq K v^\eps \iff	|\hat{\sigma} - \hat{s}| \leq v^{-\frac{4}{3}+\eps} K.
	\end{align}
	Now partition $[0, K]$ into at most $K/(v^{-\frac{4}{3}+\eps} K) + 1 \leq v^{\frac{4}{3}}$ many intervals of length at most $v^{-\frac{4}{3}+\eps} K$. Then by applying the estimate above to the midpoints of these intervals, the probability that each interval will contain $\hat{\sigma}$ for some atom $\sigma \in \wc{\cN}$ is at least $1 - C v^{\frac{4}{3}} e^{-c K v^\eps}$. This can be made arbitrarily close to $1$ uniformly in $v$ by choosing $K$ large enough. Finally we invoke the Union Bound to ensure that all conditions hold simultaneously with as high probably as desired.
\end{proof}

Let us now turn to the proof of Proposition~\ref{2nd_result_quantitative}. Here we rely on two lemmas which control the size of the {\em deep} extremal level sets of BBM, namely at height at least $m_s - v$ where $v \leq s^{1-o(1)}$ and $s$ is the time parameter. While these are fairly standard by now, for the sake of completeness, we shall include a brief proof in the end of this section. The first is an upper bound on the first moment.
\begin{lem}
\label{1st_moment} There exists constants $C < \infty$ such that for $s \geq 1$ and $0\leq v \leq s$,
	\begin{equation}
	\label{e:105.17}
		\bbE\, \cE_s([-v, \infty)) \leq C \, s e^{\sqrt{2} v - \frac{v^2}{2s} } \,.
	\end{equation}
\end{lem}
The second is a coarse lower bound with high probability. 
\begin{lem}
\label{lower_bound_2} There exists constants $c, C > 0$ such that for $s \geq 1, v \leq \frac{s}{1 + \log s}, 0\leq r \leq v/2$,
	\begin{equation}
		 \bbP\Big(\cE_s([-v, \infty)) \geq  c e^{-Cr} e^{\sqrt{2} v - \frac{v^2}{2s}}  \Big) \geq 1 - C e^{-cr}.
	\end{equation}
\end{lem}

We can now give
\begin{proof}[Proof of Proposition~\ref{2nd_result_quantitative}]
Throughout this proof we shall abbreviate $\ol{\bbP}(-) \equiv \bbP(-|\wc{W}, \wc{\cN}, \wc{Y})$. We also fix $\epsilon$, $K$ and $\Lambda$, let $v \geq K$ and ``work'' on the event $\cV^{(v)}(\eps, K, \Lambda)$.
For the upper bound, we use a first moment bound and for all $v$ control
\begin{equation}
\label{e:5.16}
\ol{\bbE} \wc{\cE}^{(s)}_{s,\wc{W}_s}
= \bbE \Big[\cE_s([-v, \infty) - \wc{W}_s) \,\Big|\, h_s^* \leq -\wc{W}_s\Big] \,,
\end{equation}
with $\cE$ as in~\eqref{e:O.2}. (Here and after, whenever used inside the expectation $\bbE$ or to define events of which the $\bbP$-probability is evaluated, the quantity $\wc{W}_s$ is treated as deterministic constant.)

Conditions~\ref{itm:1},~\ref{itm:2.5} and~\ref{itm:3} and Part~\ref{itm:b} in the definition of $\wc{Y}$, imply that $\wc{W}$ is uniformly upper bounded by $\Lambda$. It then follows from the tightness of the centered maximum, that we may drop the conditioning on the right hand side of~\eqref{e:5.16}, with the unconditional mean being an upper bound up to a harmless multiplicative constant, which depends only on $\Lambda$.

Now, whenever $s \leq v$, we trivially bound 
\begin{equation}
\label{e:105.18}
	\bbE \big[\cE_s([-v, \infty) - \wc{W}_s)\big] \leq \bbE|L_s| = \rme^s \leq \rme^v \,.
\end{equation}
Otherwise, by~\eqref{e:105.1},~\eqref{Y_rescaled_def} and~\eqref{eq_Bessel} and Lemma~\ref{1st_moment}, the last mean is at most 
\begin{equation} \label{eq_1st_moment_v}
\begin{split}
s^{C} \exp \Big(\sqrt{2} v - \big(\sqrt{2} -  \frac{v}{s}\big) & Y_s - \frac{v^2}{2s}\Big)  
	= (\hat{s} v)^{C} \exp \bigg(\sqrt{2} v - v^{\frac{2}{3}} \big(\sqrt{2} - v^{-\frac{1}{3}}\hat{s}^{-1} \big) Y^{(v)}_{\hat{s}} - \frac{v^{\frac{2}{3}}}{2\hat{s}}\bigg) \\
	& = (\hat{s} v)^{C} \exp \bigg(\sqrt{2}v-v^{\frac{2}{3}} \bigg(1 -  \frac{v^{-\frac{1}{3}} \hat{s}^{-1}}{\sqrt{2}} \bigg) \Big(\sqrt{2} Y^{(v)}_{\hat{s}}  + \frac{1}{2\hat{s}}\Big) - \frac{v^{\frac{1}{3}}}{2\sqrt{2}}\bigg) \,.
\end{split}
\end{equation}
For $s \in [v, \eps v^{\frac{4}{3}}]$, we simply drop the $Y_s$ term in \eqref{eq_1st_moment_v} and obtain the uniform bound
	\begin{align}
		\bbE \,\cE_s\big([-v, \infty) - \wc{W}_s\big)  \leq v^{C} \exp\left(\sqrt{2} v- (2\eps)^{-1} v^{\frac{2}{3}}\right). \label{eq:s_leq_v^4/3}
	\end{align}
For $s \in [\eps v^{\frac{4}{3}}, K v^{\frac{4}{3}}]$, we have $\hat s \in [\eps, K]$,  so that by~\eqref{chi_rescaled_def}, ~\eqref{range_chi_v} and the r.h.s. of~\eqref{eq_1st_moment_v},
	\begin{equation}
		\bbE \,\cE_s\big([-v, \infty) - \wc{W}_s\big) \leq \exp \left(\sqrt{2} v -v^{\frac{2}{3}}  \zeta^{(v)} +  C(\eps^{-2}v^\frac{1}{3})\right) \label{eq:s_leq_K_v^4/3} \,,
	\end{equation}
Finally, for $s \geq K v^{\frac{4}{3}}$, or $\hat s \geq K$, we drop the last term in the exponent of the middle expression in~\eqref{eq_1st_moment_v} to get
	\begin{align}
		\bbE \,\cE_s\big([-v, \infty) - \wc{W}_s\big) \leq \exp \left(\sqrt{2} v - \tfrac12 v^{\frac{2}{3}} \hat{s}^{\frac{1}{2}-\eps} \right) \label{eq:s_geq_K_v^4/3} \,,
	\end{align}
where we used the lower bound in~\eqref{eq_D-E_hat}.
Observe that for $\epsilon$ small enough and all $v$ large enough, the bounds in~\eqref{e:105.18},~\eqref{eq:s_leq_v^4/3} are dominated by that
in~\eqref{eq:s_leq_K_v^4/3}, thanks to~\eqref{range_chi_v}. 

Plugging these estimates in~\eqref{eq:C_limit_r_to_infty} we thus get
\begin{equation}\label{e:upper_bound_final_integral_calculation}
\begin{split}
	\rme^{-\sqrt{2}v} \ol{\bbE} \wc{\cC}\big([-v ,\infty)\big)
	& \leq \exp\big(-v^{\frac{2}{3}}  \zeta^{(v)} +  C(\eps^{-2}v^\frac{1}{3})\big)
		\wc{\cN}([0, Kv^{4/3}]) 
		+ \int_{Kv^{4/3}}^{\infty} \rme^{-\tfrac12 v^{\frac{2}{3}} \hat{s}^{\frac{1}{2}-\eps}} \wc{\cN}(ds) \\
	& \leq \exp\big(-v^{\frac{2}{3}}  \zeta^{(v)} +  C'(\eps^{-2}v^\frac{1}{3})\big)
	+ 3\sum_{k = K}^\infty \exp\left(-v^{\frac{2}{3}}  k^{\frac{1}{2}-\eps} \right) \times k v^{\frac{4}{3}} \\
	& \leq \exp\big(-v^{\frac{2}{3}}  \zeta^{(v)} +  C'(\eps^{-2}v^\frac{1}{3})\big) \,.
\end{split}
\end{equation}
The desire upper bound now follows from Markov's inequality.

Turning to the lower bound, by Conditions~\ref{itm:5} and~\ref{itm:6}, there exists $\hat{s}_* \in [\delta_\eps, K]$ such that
\begin{equation}
	\zeta^{(v)} = \sqrt{2}\, Y^{(v)}_{\hat{s}_*} + \frac{1}{2\hat{s}_*} \,,
\end{equation} 
and some $\sigma_* \in \wc{\cN}$ such that
\begin{equation}
	|\hat{\sigma}_* - \hat{s}_*| \leq K v^{-\frac{4}{3}+\eps}.
\end{equation} 
By the H\"older continuity \ref{itm:3} of $Y$ on $[0, K]$, this implies
\begin{equation}\label{eq:Bessel_minimized}
\sqrt{2}\, Y^{(v)}_{\hat{\sigma}_*} + \frac{1}{2\hat{\sigma}_*} \leq \zeta^{(v)} + \sqrt{2} \Lambda \left(K v^{-\frac{4}{3}+\eps}\right)^{\frac{1}{2}-\eps} + \frac{K v^{-\frac{4}{3}+\eps}}{2\delta_\eps^2}.
\end{equation}

Observe that $s_* := v^{4/3} \hat{s}_*$ and hence $\sigma_* := v^{4/3} \hat{\sigma}_*$ can become arbitrarily large by choosing $v$ large enough. It thus follows from the tightness of the centered maximum and Lemma~\ref{lower_bound_2} (with $r = v^{\frac{1}{3}}$) that for all such $v$,
\begin{equation}\label{eq:lower_bound_C_r}
	\wc{\cE}^{(\sigma_*)}_{\sigma_*}\big([-v, \infty) - \wc{W}_{\sigma_*}\big) \geq  \exp\left(\sqrt{2} (v+\wc{W}_{\sigma_*}) - \frac{(v+\wc{W}_{\sigma_*})^2}{2\sigma_*} - C v^{\frac{1}{3}}\right)
\end{equation}
occurs with $\ol{\bbP}$ probability at least $1 - C e^{-cv^{\frac{1}{3}}}$ for $v$ large enough. As in~\eqref{eq_1st_moment_v} the right hand side is at least 
\begin{equation}
\exp \Big(-\sqrt{2}v-v^{\frac{2}{3}} \Big(\sqrt{2} Y_{\hat{\sigma}_*}^{(v)}  + \frac{1}{2\hat{\sigma}_*} \Big) + O(\eps^{-2} v^\frac{1}{3}) \Big)
\geq 
\exp \Big(-\sqrt{2}v -v^{\frac{2}{3}} \zeta^{(v)} -\epsilon v^{2/3}\Big) \,,
\end{equation}
for all $v$ large enough,  where we have used~\eqref{eq:Bessel_minimized} for the second inequality. 
It remains to observe that, in view of~\eqref{eq:C_limit_r_to_infty}, the left hand side above is also a lower bound on $\wc{\cC}([-v, 0])$.
\end{proof}

\subsection{Control of the deep extreme level sets}
In this subsection we prove Lemma~\ref{1st_moment} and Lemma~\ref{lower_bound_2}. As these arguments are standard, we allow ourselves to be brief. The proof of the first of the two lemmas is immediate.
\begin{proof}[Proof of Lemma~\ref{1st_moment}]
By the Many-To-One Lemma~\ref{l:4.1} and the Gaussian tail formula, the desired mean is equal to 
\begin{equation}
\label{e:105.31}
	\rme^s \times \frac{1}{\sqrt{2\pi}} \frac{\sqrt{s}}{m_s-v} \rme^{-\frac{(m_s - v)^2}{2s}} \,,
\end{equation}
which is bounded by the right hand side of~\eqref{e:105.17} by a simple computation.
\end{proof}

To obtain a lower bound with high probability, we employ a severe (but sufficient for our purposes) truncation. Explicitly, in place of $\cE_s$ we define
\begin{equation}
	\wt{\cE}_s := \sum_{x \in L_s} \delta_{h_s(x) - m_s} 
		1_{\{(h_r(x_r))_{r \leq s} \in J_s\}} \,,
\end{equation}
where we recall that $x_r$ stands for the ancestor of $x \in L_s$ at time $r \leq s$, and
	\begin{equation}
		J_s := \Big\{w \in \bbR^{[0,s]} :\: w_r \leq \frac{r}{s}w_s + 1 \,,\,\, r \in [0,s] \Big\} \,.
	\end{equation}

The next lemma gives lower and upper bounds on the moments of the level sets of $\wt{\cE}_s$. 
\begin{lem}
\label{l:5.6}
For all $s \geq 1$, $0\leq v \leq \frac{s}{1+\log s}$, 
\begin{equation}
\label{e:105.34}
	\bbE \wt{\cE}_s([-v, \infty)) \asymp e^{\sqrt{2} v - \frac{v^2}{2s} }
\end{equation}
and
\begin{equation}
\label{e:105.35}
	\bbE \wt{\cE}_s([-v, \infty))^2 \leq C \Big(\bbE \wt{\cE}_s\big([-v, \infty)\big)\Big)^2 \,.
\end{equation}
\end{lem}
\begin{proof}
By Lemma~\ref{l:102.1} and standard tilting for BM, for any Borel set $A \subset \bbR$, 
\begin{equation}
\label{e:105.36a}
\begin{split}
	\bbP \big(W_s \in A\,,\,\,W_{[0,s]} \in J_s\big)
	& = \int_A \bbP(W_s \in \rmd w)
	\bbP \big(\max W_{[0,s]} \leq 1 \,\big|\, W_s = 0\big) \\
	& = \frac{2}{s} \bbP (W_s \in A) \,.
\end{split}
\end{equation}
Thus, together with the Gaussian tail formula and a similar calculation as in~\eqref{e:105.31} we get~\eqref{e:105.34}. Observe that by the Gaussian tail formula, we get the same lower bound (up to a smaller constant) also for the first moment of $\wt{\cE}_s([-v, -v+1])$, so that
\begin{equation}
\label{e:105.36b}
		\bbE \wt{\cE}_s([-v, \infty)) \asymp \bbE \wt{\cE}_s([-v, -v+1]) \gtrsim 1 \,.
\end{equation}

Turning to the second moment, by the Many-to-Two Lemma (c.f.~\cite{CHL19} for a similar computation), for any Borel set $A \subset \bbR$, the second moment of $\wt{\cE}_s(A - m_s)$ is equal to
\begin{equation}
\label{e:105.36}
\begin{split}
	\big(\bbE & \wt{\cE}_s(A-m_s)\big) + \\
 	 & 2 e^{s} \int_0^s e^{s-\tau} 
 	 	\bbP\Big(h_s(x), h_s(x') \in A,\, \big(h_r(x_r)\big)_{r \leq s}, \big(h_r(x'_r)\big)_{r \leq s} \in J_s \,\Big|\, 
 	 		x, x' \in L_s,\, |x \wedge x'| = \tau \Big)
 	 	\Big) \rmd \tau 
 \end{split}
\end{equation}
where we recall that $|x\wedge x'|$ stands for the splitting time of $x$ and $x'$. 

Next we specialize to $A = [u,u+1]$ for $u \in [s/10, \sqrt{2} s]$, and condition on $h_\tau(x_\tau) = h_\tau(x'_\tau)$ to upper bound the probability in the last integral by
\begin{equation}
\label{e:105.38}
	\bbE \Big(f_{\tau}\big(W_\tau- u_\tau \big) ;\; W_s \in [u,u+1]\,,\,\,W_{[0,s]} \in J_s\big) \,,
\end{equation}
where $u_r := \frac{r}{s} u$, and 
\begin{equation}
	f_{\tau}(z) := \bbP \Big(W_{s-\tau} \geq u_{s-\tau}\,,\,\,
			\max_{r \leq s-\tau} \big(W_r - u_r\big) \leq 2\big)\,\Big|\, 
			W_0 = z\Big) \,.
\end{equation}
Let us first consider the case $\tau \leq s-1$. Proceeding as in~\eqref{e:105.36a},
we write $f_\tau(z)$ as
\begin{equation}
\int_{w \geq 0}  \bbP \big(W_{s-\tau} - u_{s-\tau} \in \rmd w \,\big|\, W_0=z\big)\,
\bbP \big(\max W_{[0,s-\tau]} \leq 2 \,\big|\, W_0 = z, W_{s-\tau} = w\big) \,.
\end{equation}
The second term is at most $2\cdot (0-2)^- \cdot (z-2)^-/(s-\tau)$ by Lemma~\ref{l:102.1} as $w \in [0, 2]$, while the conditioning $W_0 = z$ in the first term may be changed to $W_0 = 0$ at the cost of a factor, which is at most $C\exp\big(-uz^-/s \big)$ by explicit computation, thanks to the conditions $z \leq 2$, $s-\tau \geq 1$, $w\in [0, 2]$, and the upper bound on $u$. In view of the lower bound on $u$, we thus get
\begin{equation}
	f_\tau(z) \leq C\rme^{cz^-}(z-2)^- \bbP \big(W_{s-\tau} \in  [u_{s-\tau},u_{s-\tau} + 1]\,,\,\,W_{[0,s-\tau]} \in J_{s-\tau}\big) \,.
\end{equation}
Plugging this in~\eqref{e:105.38} and proceeding as in~\eqref{e:105.36a} by conditioning both on $W_s$ and $W_{\tau}$, 
we get as an upper bound $\bbP \big(W_{s-\tau} \in  [u_{s-\tau},u_{s-\tau} + 1]\,,\,\,W_{[0,s-\tau]} \in J_{s-\tau}\big)$ times
\begin{multline}
\label{e:105.43a}
 C \bbP \big(W_s \in [u,u+1]\big) (\tau+1)^{-1} (s-\tau+1)^{-1} \bbE \big(\rme^{-cW_\tau^-} ((W_\tau^-)^3+1) \,\big|\, W_s = 0\big)  \\
\leq C (\tau \wedge (s-\tau) + 1)^{-3/2} \bbP \big(W_s \in [u,u+1]\,,\,\,W_{[0,s]} \in J_s\big) \,,
\end{multline}
where we have used that the expectation above is at most $C (\tau \wedge (s-\tau) + 1)^{-1/2}$ by direct computation. When $\tau \in [s,s-1]$, we can simply drop $f_\tau$ from the expectation in~\eqref{e:105.38} and upper bound the mean in~\eqref{e:105.38} just by the right hand side of~\eqref{e:105.43a}.

Using these bounds, the second term of~\eqref{e:105.36} is at most
\begin{equation}
\label{e:105.43}
 C \bbE \wt{\cE}_s([u,u+1]-m_s)\big)
 	\int_0^s (\tau \wedge (s-\tau) + 1)^{-3/2}
	 	\big(1+\bbE \wt{\cE}_{s-\tau} \big([u_{s-\tau}, u_{s-\tau} + 1] -m_{s-\tau}\big)\big) \rmd \tau
\end{equation}
Plugging in $u=m_s - v$, observing that $u_r \geq m_r - v_r - C$ and using the first moment upper bound, the integral is at most
\begin{equation}
	C \int_0^s (\tau \wedge (s-\tau) + 1)^{-3/2} \exp \Big((s-\tau)\Big(\sqrt{2}\frac{v}{s} - \frac{v^2}{2s^2}\Big)\Big) \rmd \tau
\leq  C' \bbE \wt{\cE}_s([-v,-v+1])\big) \,,
\end{equation}
where we have again used the bound on $v$. Together with the first mean in~\eqref{e:105.43} at $u = m_s - v$ and~\eqref{e:105.36b} this shows the desired claim.
\end{proof}

We can now give
\begin{proof}[Proof of Lemma~\ref{lower_bound_2}]
Thanks to Lemma~\ref{l:5.6} and the Paley-Zygmund inequality there exists $\delta > 0$ such that for all sufficiently large $s \geq 1$ and all $v \leq \frac{s}{1+\log s}$
\begin{equation}\label{e:Paley-Zygmund-lower-bound}
	\bbP\Big(\cE_s([-v, \infty)) \geq \delta e^{\sqrt{2} v - \frac{v^2}{2s}}  \Big) 
	\geq \bbP\Big(\wt{\cE}_s([-v, \infty)) \geq \delta e^{\sqrt{2} v - \frac{v^2}{2s}} \Big) 
	\geq \delta  \,.
\end{equation}
To get a high probability event, we condition on generation $r > 0$. Then, since the reproduction rate is such that $\bbE|L_r| = \rme^r$ and the motion is such that $m_r \leq \sqrt{2} r$, it follows by standard BBM facts (namely, the geometric distribution of $L_r$ and the tightness of centered \emph{minimum} of $h_r$) that the probability of the event
\begin{equation}
	\cA_r : = \Big\{|L_r| \geq \rme^{r/2} \,,\, \min_{L_r} h_r \geq -2r\Big\} 
\end{equation}
is at least $1 - C e^{-cr}$ for some constants $c, C > 0$. At the same time by the Markov property of BBM, for $\frac{s}{1 + \log s} \geq v \geq 2r$,
\begin{equation}
\bbP \Big(\cE_s([-v, \infty)) \geq z
\, \Big|\, \cA_r \Big) 
\geq 1 - \bbP \big(\cE_{s-r}([-v+2r,\, \infty)) < z\big)^{e^{r/2}} \, ,
\end{equation}
which is at least $1 - (1-\delta)^{e^{r/2}}$ by the estimate \eqref{e:Paley-Zygmund-lower-bound} if in the above we set $z := \delta e^{\sqrt{2} (v-2r) - \frac{(v-2r)^2}{2(s-r)}}$. The result now follows by observing $\frac{(v-2r)^2}{2(s-r)} \leq \frac{v^2}{2s}$ since $s \geq \frac{v}{2}$. 
\end{proof}

\subsection{Almost sure growth for the cluster level sets}
Lastly let us give
\begin{proof}[Proof of Proposition~\ref{p:2}]
In the following we let $\cV(\eps, K, \Lambda) := \cV^{(1)}(\eps, K, \Lambda)$. In light of Proposition \ref{p:nice_event_whp}, it suffices to show that for arbitrarily small $\eps > 0$ and large $K, \Lambda < \infty$, as $v \to \infty$, 
\begin{equation}
\frac{\log \wc{\cC}([-v, 0]) -\sqrt{2} v }{v} \to 0 \,,
\end{equation}
almost surely on $\cV(\eps, K, \Lambda)$. Because of the monotonicity of $\log \wc{\cC}([-v, 0]) $ in $v$, this will follow from a standard Borel-Cantelli argument (along integer $v$), once we have the following quantitative version of Proposition \ref{2nd_result_quantitative}:
\begin{equation}
	\bbP \bigg(\bigg|\frac{\log \wc{\cC}([-v, 0]) -\sqrt{2} v }{ v^{\frac{2}{3}} \log \log v}\bigg| > \epsilon \,\bigg|\, \wc{W}, \wc{\cN}, \wc{Y}  \bigg) < C e^{-cv^{\frac{1}{3}}} \,,
\end{equation}
for all large $v$, uniformly in $(\wc{W}, \wc{\cN}, \wc{Y} ) \in \cV(\eps, K, \Lambda)$. \medskip

Starting with the upper bound, as before we have the bound \eqref{e:105.18} for $s\leq v$ and for $s\geq v$, we use the estimate \eqref{eq_1st_moment_v} with Condition \ref{itm:2} to get
\begin{equation}
	\bbE \big[\cE_s([-v, \infty) - \wc{W}_s)\big] \leq s^C \exp \Big(\sqrt{2} v - \big(\sqrt{2} -  1\big)  s^{\frac{1}{2}-\eps} - \frac{v^2}{2s} \Big) \,,
\end{equation}
which is at most $v^C \exp \Big(\sqrt{2} v - \frac12 	v^{\frac{2}{3}} \Big)$ for $s \in [v, v^{\frac{4}{3}}]$ and at most $s^C \exp \Big(\sqrt{2} v - 0.4  s^{\frac{1}{2}-\eps} \Big)$ for $s\geq v^{\frac{4}{3}}$. Summing over all $s\in \wc{\cN}$ as in \eqref{e:upper_bound_final_integral_calculation}, we get the overall estimate
\begin{equation}
\ol{\bbE} \wc{\cC}\big([-v ,\infty)\big) \leq v^{C'} \exp \Big(\sqrt{2}v - v^{\frac{2}{3}-2\eps} \Big)
\end{equation}	and the desired upper bound now follows from Markov's inequality.\medskip

Turning to the lower bound, by Part 3 of Theorem \ref{p:1}, with probability at least $1 - C e^{-cv}$ there exists $\sigma_* \in \wc{\cN}$ satisfying $|\sigma_*-v^{\frac{4}{3}}|\leq v$ . As before, it follows from the tightness of the centered maximum and Lemma~\ref{lower_bound_2} that for $v$ large enough,
\begin{equation}\label{eq:lower_bound}
\wc{\cE}_{\sigma_*}([-v, \infty) - \wc{W}_{\sigma_*}) \geq \exp\Big(\sqrt{2} (v + \wc{W}_{\sigma_*}) - \frac{(v + \wc{W}_{\sigma_*})^2}{2s_*} - C v^{\frac{1}{3}} \Big) 
\end{equation}
occurs with probability at least $1 - Ce^{-cv^{\frac{1}{3}}}$. Using the estimate \eqref{eq_1st_moment_v} with Condition \ref{itm:2}, the right side is at least
\begin{equation}
	 \exp \Big(\sqrt{2} v - \sqrt{2}\cdot 3  (\sigma_* \log \log \sigma_*)^{\frac{1}{2}} - \frac{v^2}{2\sigma_*} - C v^{\frac{1}{3}}\Big) \geq \exp \Big(\sqrt{2} v - C' v^{\frac{2}{3}} \sqrt{\log \log v} \Big) \,,
\end{equation}
for $v$ large enough. Since this is also a lower bound on $\wc{\cC}([-v, 0])$ in view of~\eqref{eq:C_limit_r_to_infty}, the conclusion follows.
\end{proof}

\section{Strong representation for the cluster law}
\label{s:4}
In this section we prove Theorem~\ref{p:1}. The proof relies
on the weak representation of the cluster law from Subsection~\ref{s:2.1}, via the DRW $(\wh{W}, \cN, H)$ from~\eqref{e:102.27} conditioned on the Ballot event $\cA_t$ from~\eqref{e:102.7}.

The proof will use three auxiliary lemmas. The first lemma which shows that 
the law of $(\wh{W}, \cN)$, restricted to any finite time interval, tends to a limit in Total Variation as $t \to \infty$.
\begin{lem}\label{lem:TV_backbone_timestamps} 
There exists a joint process $(\wc{W}, \wc{\cN}) \in C([0, \infty)) \times \bbM_p([0, \infty))$ such that for every fixed $r > 0$, 
	\begin{equation}
	\label{e:104.1}
			\Big\|\bbP_{0,0}^{t,0}\big((\wh{W}_{t, [0,r]}, \cN_{[0,r]}) \in \cdot \,\big|\, \cA_t \big) - 
			\bbP\big((\wc{W}_{[0, r]}, \wc{\cN}_{[0, r]}) \in \cdot \big) \Big\|_{\rm TV}
			\underset{t \to \infty}\longrightarrow 0 \,.
	\end{equation}
\end{lem}

The second lemma shows that the law of $\wh{W}$ under $\cA_t$ tends to be that of Bessel-3 away from origin and after taking first $t \to \infty$. Again, this convergence is in Total Variation.
\begin{lem}
\label{l:4.2}
Let $(Y_s)_{s \geq 0}$ a Bessel-3 and $y_0 \geq 0$. Set $\wc{Y}_s := -Y_s - \frac{3}{2\sqrt{2}} \log^+s$. Then,
\begin{equation}\label{e:2.13}
\lim_{r \to \infty}
\sup_{R \geq r}
\limsup_{t \to \infty}
\Big\|\bbP_{0,0}^{t, 0} \big(\wh{W}_{t,[r,R]} \in \cdot \,\big|\cA_t \big) - 
	\bbP \big(\wc{Y}_{[r,R]} \in \cdot \,\Big|\, \wc{Y}_0 = y_0 \big)\Big\|_{\rm TV} = 0 \,.
\end{equation}
\end{lem}

The third lemma shows that the law of $\cN$ does not change too much under $\cA_t$.
\begin{lem}
	\label{l:3.2a}
	There exists $C < \infty$ such that for any $s,r \geq 0$ and non-negative measurable test function $\varphi : \bbM_p([s,s+r]) \to \bbR_+$, 
	\begin{equation}
		\label{e:103.6}
		\limsup_{t \to \infty} \bbE_{0,0}^{t,0} \big(\varphi(\cN_{[s,s+r]}) \,\big|\, \cA_t \big) \leq C\Big(\sqrt{\frac{r}{s\vee 1}} + 1\Big) \, \bbE \big(\varphi(\cN_{[s,s+r]})\big)
	\end{equation}
\end{lem}

The proofs of these auxiliary lemmas will be given in Subsection \ref{s:cluster_law_rep_aux_lem}, before which we can give:
\begin{proof}[Proof of Theorem~\ref{p:1}] 
It is standard that one can define a probability space which includes both the process $(\wc{W}, \wc{\cN})$ from Lemma~\ref{lem:TV_backbone_timestamps} and a process $\wc{\cE} = (\wc{\cE}_{s,y})_{s,y}$ which is independent of $(\wc{W}, \wc{\cN})$ and whose marginals are independent and obey~\eqref{e:101.12}. We wish to show that the process $(\wc{W}, \wc{\cN}, \wc{\cE})$ satisfies all the requirements in the theorem. To avoid confusion, let us denote the probability measure on this space by $\wc{\bbP}$ and expectation by $\wc{\bbE}$.
		
		By the independence assumptions, we clearly have
		\begin{equation}
			\label{e:4.28}
			\bbE \bigg( \varphi \Big(\int_0^r \! \wc{\cE}_{s, \wc{W}_s}\wc{\cN}(\rmd s)\Big)\, \bigg|\,\, \wc{W}, \wc{\cN}\bigg) = 
			\wc{\bbE} \Phi_{r}(\wc{W}_{[0,r]}, \wc{\cN}_{[0,r]})\,,
		\end{equation}
		where
	$\Phi_r:\: C([0,r]) \times \bbM_p([0,r]) \to \bbR$ is defined via
		\begin{equation}
			\Phi_r(w, \eta) := \wc{\bbE} \, \varphi \Big(\int_0^r \wc{\cE}_{s, w_s} \eta (\rmd s)\Big) \,,
		\end{equation}
	for a (bounded measurable) test function $\varphi: \bbM_p((-\infty, 0]) \to \bbR$. 
		
		At the same time, recalling the DRW from Subsection~\ref{s:2.1.2}, we also let $\ol{\bbP}_t$ be the conditional measure $\bbP_{0,0}^{t,0}(-|\cA_t)$, with $\ol{\bbE}_t$ the associated expectation. Then, by independence of the components of the DRW and the independence of the BBMs in the component $H$ thereof, we also clearly have
		\begin{equation}
			\label{e:4.26}
			\ol{\bbE}_t \bigg( \varphi \Big(\int_0^r \! \cE_{s}^{s} \big(\cdot - \wh{W}_{t,s} \big)\cN(\rmd s)\Big)\, \bigg|\,\, \wh{W}, \cN\bigg) = 
			\ol{\bbE}_t \Phi_{r}(\wh{W}_{t,[0,r]}, \cN_{[0,r]})\,,
		\end{equation}
		with $\Phi_r$ as above.

		Now, if $\varphi$ is bounded and continuous (w.r.t.~the vague topology on $\bbM_p((-\infty, 0])$), then 
		by Lemma~\ref{l:7.0} and Lemma~\ref{l:5.2} the left hand side
		of~\eqref{e:4.26} converges to $\int \varphi(\omega) \nu_r(\rmd \omega)$, where $\nu_r$ is as in Lemma~\ref{l:7.0}, while by Lemma~\ref{lem:TV_backbone_timestamps}, the right hand side of~\eqref{e:4.26} tends to the left hand side of~\eqref{e:4.28}. This shows that 
		\begin{equation}
			\wc{\cC}_r \sim \nu_r \,.
		\end{equation}
		
By weak convergence of $\nu_r$ to $\nu$ as given in Lemma~\ref{l:7.0}, for any $v > 0$, the collection $(\wc{\cC}_r([-v, 0]))_{r \geq 0}$ is tight. It therefore follows, by monotoncity again, that the limit
\begin{equation}
	\wc{\cC} = \lim_{r \to \infty} \wc{\cC}_r 
\end{equation}
exists almost surely as an element of $\bbM((-\infty, 0])$, both point-wise (i.e. when both sides are evaluated on any Borel subset of $(-\infty,0]$) and therefore also vaguely. Thanks to Lemma~\ref{l:7.0} again we must also have $\wc{\cC} \sim \nu$.
 This establishes the statement of Theorem \ref{p:1} up to and including Part~1.
		
For Part~2, it follows from Lemma~\ref{lem:TV_backbone_timestamps} and Lemma~\ref{l:4.2} that
		\begin{equation}
			\label{e:4.32}
			\lim_{r \to \infty}
			\sup_{R \geq r}
			\Big\|\wc{\bbP} \big(\wc{W}_{[r,R]} \in \cdot \big) - 
			\bbP \big(\wc{Y}_{[r,R]} \in \cdot \,\Big|\, \wc{Y}_0 = y_0 \big)\Big\|_{\rm TV} = 0 \,.
		\end{equation}
	Now, by standard measure theoretic arguments, we can approximate any Borel subset of $\cC([r, \infty])$ arbitrarily well by 
	a Borel subset of $\cC([r, R])$ under 
	the law $\frac12 \wc{\bbP} \big(\wc{W}_{[r,\infty)} \in \cdot \big) + \frac12 \bbP \big(\wc{Y}_{[r,\infty)} \in \cdot \,\Big|\, \wc{Y}_0 = y_0 \big)$ as long as $R$ is chosen large enough. This implies that the same Borel subset of $\cC([r, R])$ can be chosen for the approximation under each of these laws separately. Together with the uniform (in $R$) convergence in~\eqref{e:4.32}. This shows~\eqref{e:2.3}.
	
	Lastly, for Part~3, thanks to Lemma~\ref{l:3.2a} and the i.i.d. $\text{Exp}(2)$ law of the distance between succeeding atoms of $\cN$, we have
	\begin{equation}
		\limsup_{t \to \infty} 	\bbP_{0,0}^{t,0} \big(\cN\big([s-u/2,s]\big) = 0 \,\big|\, \cA_t \big) \leq C(\sqrt{u/2}+1) \rme^{-u} \leq C \rme^{-cu} 
	\end{equation}
	with a similar bound holding for the limit-superior as $t \to \infty$ of $\bbP_{0,0}^{t,0} \big(\cN\big([s,s+u/2]\big) = 0 \,\big|\, \cA_t \big)$. 
	Thanks to Lemma~\ref{lem:TV_backbone_timestamps} and the Union Bound, we thus get the first inequality in~\eqref{e:101.14}.
	The second inequality in~\eqref{e:101.14} follows in the same way thanks to the exponential decay of $\bbP(\cN\big([0,s]\big) > (2+\epsilon) s\big)$ in $s$ for any $\epsilon > 0$, Lemma~\ref{l:3.2a} and Lemma~\ref{lem:TV_backbone_timestamps}.
\end{proof}

\subsection{Proofs of auxiliary lemmas}\label{s:cluster_law_rep_aux_lem}

Let us first give the proof of Lemma \ref{lem:TV_backbone_timestamps}. This is similar to the proof of Lemma \ref{l:2.2}, replacing the Ballot estimate for Brownian motion with that for the DRW. 

\begin{proof}[Proof of Lemma \ref{lem:TV_backbone_timestamps}]
Define $\ol{W}_{s} := W_s - \frac{3}{2\sqrt{2}} \log^+ s$. Note that
\begin{equation}
\wh{W}_{t,s} = W_s - \gamma_{t, s} = \ol{W}_s + d_t s
\end{equation}
where $d_t := \frac{3}{2\sqrt{2}} \frac{\log^+ t}{t}$. 
Using Bayes' rule, the Markov property, and the Cameron-Martin formula, we have
\begin{equation}
\begin{split}
\label{e:2.8}
\bbP_{0, 0}^{t, 0} \left(\wh{W}_{t, [0,r]} \in \rmd w, \cN_{[0, r]} \in \rmd \eta \, \Big| \cA_t \right) &= \bbP_{0, 0} \left(\wh{W}_{t, [0,r]} \in \rmd w, \cN_{[0, r]} \in \rmd \eta\right) g^{(1)}(w, \eta)\,  g^{(2)}_{t, r}(w) \\
 &= \bbP_{0, 0} \left(\ol{W}_{[0,r]} \in \rmd w, \cN_{[0, r]} \in \rmd \eta\right) 
 g^{(1)}(w, \eta)\,g^{(2)}_{t, r}(w)\,  g^{(3)}_{t,r}(w)
\end{split}
\end{equation}
with
\begin{equation}
g^{(1)}(w, \eta) := \bbP \Big(\max_{s\in \eta: \, s \in[0, r]} w_s + \wh{h}^{(s) \, *}_s \leq 0 \Big) \,,
\end{equation}
\begin{equation}
g^{(2)}_{t, r}(w) := \frac{p_{t-r}(w_r+ \gamma_{t, r}, 0)}{p_t(0, 0)} \cdot \frac{\bbP^{t, 0}_{r, w_r}(\cA_{t}(r,t))}{\bbP^{t, 0}_{0, 0}(\cA_t)} \,,
\end{equation}
where $p_t(x, y)$ the transition density of BM, and
\begin{equation}
g^{(3)}_{t,r}(w) := 
	 \exp \Big(d_t \Big(w_r + \frac{3}{2\sqrt{2}} \log^+ r\Big) - \frac{1}{2} d_t^2 r\Big) \,. 
\end{equation}

Now, observe that for $w$ and $r$ fixed, $g^{(3)}_{t,r}(w)$ converges to $1$ as $t \to \infty$ and bounded by $2 e^{|w_r|}$ for large $t$, because $d_t \to 0$ as $t\to \infty$.
At the same time, by Lemma~\ref{lem:15}, we have that as $t \to \infty$,
\begin{equation}
\frac{\bbP^{t, 0}_{r, z}(\cA_{t}(r, t))}{\bbP^{t, 0}_{0, 0}(\cA_t)} = \frac{f^{(r)}(z)}{f^{(0)}(0)} (1+o(1)) \,,
\end{equation}
and is uniformly bounded by $O(1) (z^- + 1)$.
Lastly, by direct computation, we have that 
\begin{align}
	\frac{p_{t-r}(z + \gamma_{t, r}, 0)}{p_t(0, 0)} = \frac{1}{\sqrt{1 - \frac{r}{t}}} \exp\left(-\frac{(z + \gamma_{t,r})^2}{2(t-r)}\right) = 1+o(1)  \,,
\end{align}
as $t\to \infty$, and is uniformly bounded. 

Combining these limits, we get 
\begin{equation}\label{eq:BM_conditioned_2}
g^{(1)}_{t,r}(w,\eta) \times g^{(2)}_{t, r}(w) \times g^{(3)}_{t,r}(w) \underset{t \to \infty}\longrightarrow h_r(w,\eta) := 
g^{(1)}_{t,r}(w,\eta) \frac{f^{(r)}(w_r)}{f^{(0)}(0)}
\end{equation}
with the left hand side uniformly bounded by $O(1) (|w_r| + 1) e^{|w_r|}$ (which is $\bbP_{0,0}$-integrable). It then follows from Dominated Convergence that the convergence in~\eqref{eq:BM_conditioned_2} holds also in $L^1(\bbP_{0,0})$. This shows~\eqref{e:104.1} with 
$(\wc{W}_{[0, r]}, \wc{\cN}_{[0, r]})$ replaced by 
the process $(\wc{W}^{(r)}, \wc{\cN}^{(r)})$ 
on $C([0, r]) \times \bbM_p([0, r])$ defined by
\begin{equation}
\bbP\Big(\wc{W}^{(r)} \in \rmd w, \wc{\cN}^{(r)} \in \rmd \eta\Big) = \bbP_{0,0}\Big(\ol{W}_{[0,r]} \in \rmd w,\, \cN_{[0,r]} \in \rmd \eta\Big) \cdot h_r(w, \eta)\,.
\end{equation}
Since these (limiting) processes are consistent for different $r$'s by definition, they uniquely define an infinite time process $(\wc{W}, \wc{\cN})$ on $C([0, \infty)) \times \bbM_p([0, \infty))$, of which they are finite time restrictions, and as such~\eqref{e:104.1} holds.
\end{proof}

Next we give the proof of Lemma \ref{l:4.2}. The main idea is comparing the law of $\wh{W}$ under $\cA_t$ to that of a BM conditioned to stay positive, which converges in TV to Bessel-3 by Lemma \ref{l:2.2}.

\begin{proof}[Proof of Lemma \ref{l:4.2}]
Recall that $C([r,R])$ stands for the space of continuous functions on $[r,R]$. Thanks to Lemma~\ref{l:2.13} and the lower bound on $\bbP_{0,0}^{t,0}(\cA_t)$ implied by Lemma~\ref{lem:15}, for any $0 < s < r < R < t-s$ and measurable $\varphi :\: C([r,R]) \to \bbR$ with $\|\varphi\|_\infty \leq 1$,
\begin{align}\label{e:2.10}
\bbE_{0,0}^{t, 0} \Big(\varphi\big(\wh{W}_{t,[r,R]}\big)\,\Big|\cA_t \Big) 
= \bbE_{0,0}^{t, 0} \Big(\varphi\big(\wh{W}_{t,[r,R]}\big)\,\Big| \wt{\cA}_{t}(s, t-s)\Big) + o(1) \,,
\end{align}
where $\wt{\cA}_t(s,t-s)$ is as in Lemma~\ref{l:2.13}, and the $o(1)$ term can be made arbitrarily small by taking first $s$ large enough, then choosing any $R > r > s$, and finally taking $t$ large enough, uniformly in all such $\varphi$. 
By conditioning on $\wh{W}_{t,s}, \wh{W}_{t,t-s}$, we may further replace $\varphi\big(\wh{W}_{t,[r,R]}\big)$ in the last expectation by the quantity $\Phi_{t,s}(\wh{W}_{t,s}, \wh{W}_{t,t-s})$, where 
\begin{equation}
\Phi_{t,s}(x,y) := \bbE \Big(\varphi\big(\wh{W}_{t,[r,R]}\big)\,\big|\,W_s = x,\, W_{t-s} = y, \, \max W_{[s,t-s]} \leq 0 \Big) \,,
\end{equation}
for $x,y < 0$.

Now, by Lemma~\ref{l:2.2} and since $\gamma_{t,s} = \frac{3}{2\sqrt{2}} \log^+s - d_t s$ with $d_t \to 0$ as $t \to \infty$, we have
\begin{equation}
\Phi_{t,s}(x,y) \underset{t \to \infty}\longrightarrow \bbE \Big(\varphi(\wc{Y}_{[r,R]}\big) \,\Big|\, Y_s = x^- \Big) \,,
\end{equation}
as $t \to \infty$ for fixed $x,y,s,r,R$ and uniformly in $\varphi$.
At the time, by Lemma~\ref{l:Bessel_TV} we also have
\begin{equation}
\lim_{r \to \infty} \sup_{R \geq r}
	\Big|\bbE \Big(\varphi(\wc{Y}_{[r,R]}\big) \,\Big|\, Y_s = x^- \Big) 
	- \bbE \Big(\varphi(\wc{Y}_{[r,R]}\big) \,\Big|\, Y_0 = y_0 \Big)\Big|
= 0
\end{equation}
for fixed $s,x$, uniformly in $\varphi$. Combined, this gives for fixed $s,x,y$,
\begin{equation}
\label{e:4.18}
\lim_{r \to \infty}\sup_{R\geq r} \limsup_{t \to \infty} \Big|\Phi_{t,s}(x,y) - \bbE \big(\varphi(\wc{Y}_{[r,R]}\big) \,\big|\, Y_0 = y_0 \big) \Big| = 0 \,,
\end{equation}
uniformly in $\varphi$.

Now, observe that $W_s$, $W_{t-s}$ are jointly Gaussian under $\bbP_{0,0}^{t,0}$ with covariance matrix converging to $s\cdot \text{Id}$ as $t \to \infty$ and $s$ fixed. Applying the Bayes rule and the Markov property, we have
\begin{equation}\label{e:104.17}
\begin{split}	
\bbP_{0, 0}^{t,0}\Big((W_s, W_{t-s}) \in (dx, dy)& \,\Big|\, \wt{\cA}_{t}(s, t-s) \Big) 
\\ 
&\leq \frac{\bbP_{s,x}^{t-s,y} \big(\max W_{[s,t-s]} \leq 0\big)}{\bbP_{0,0}^{t,0}\Big(\wt{\cA}_{t}(s, t-s)\Big)} \times \bbP_{0, 0}^{t,0}\big((W_s, W_{t-s}) \in (dx, dy)\big) \\
&\leq C x^-y^-  \times (s^2)^{-\frac{1}{2}} \exp\left(-C^{-1} s^{-1} (x^2+y^2) \right) dx dy \,,
\end{split}
\end{equation}
for fixed $s$ and all large $t$, where we applied the estimates of Lemma~\ref{l:102.1}, Lemma~\ref{lem:15}, and Lemma~\ref{l:2.13}. This shows that, for fixed $s$ and all large $t$, the conditional joint density of $W_s, W_{t-s}$ in \eqref{e:104.17} is bounded by a Lebesgue-integrable function on $\bbR^2$. Invoking the Dominated Convergence Theorem, we thus get
\begin{equation}
\Big|\bbE_{0,0}^{t, 0} \Big(\varphi\big(\wh{W}_{t,[r,R]}\big)\,\Big|\wt{\cA}_{t}(s, t-s)\Big) - 
	\bbE \big(\varphi(\wc{Y}_{[r,R]}\big) \,\big|\, Y_0 = y_0 \big) \Big| 
	\longrightarrow 0 \,,
\end{equation}
in the same limits under~\eqref{e:4.18}. Combined with~\eqref{e:2.10} this gives the desired statement.
\end{proof}

Lastly we give the proof of Lemma \ref{l:3.2a}.

\begin{proof}[Proof of Lemma \ref{l:3.2a}]
	Recall that the DRW components $\cN$, $\wh{W}$ and $H$ are independent 
	and that $H$ is a collection of independent processes BBMs. Recall also that $\cN$ is a PPP and therefore its restrictions to disjoint sets are independent and any given point is almost surely uncharged. We thus have
	\begin{equation}
		\begin{split}
			\label{e:103.7}
			\bbE_{0,0}^{t,0} \big(\varphi(\cN_{[s,s+r]}) ;\; \cA_t \big)
			& \leq \bbE_{0,0}^{t,0} \big(\varphi(\cN_{[s,s+r]}) ;\; \cA_t(0,s) \cap \cA_t(s+r,t)\big) \\
			& = 	\bbE \big(\varphi(\cN_{[s,s+r]}) \,
			\bbP_{0,0}^{t,0} \big(\cA_t(0,s) \cap \cA_t(s+r,t) \big) \,.
		\end{split}
	\end{equation}
	Conditioning on $\wh{W}_{t,s}, \wh{W}_{t,s+r}$ and using Lemma~\ref{lem:15},  the last probability is at most 
	\begin{equation}
		\label{e:103.8}
C (s\vee 1)^{-1} (t-s-r)^{-1} \bbE_{0,0}^{t,0} \big(\big(\wh{W}^-_{t,s} + 1) \big(\wh{W}^-_{t,s+r} + 1\big)\big) \,.
	\end{equation}
	Since $\wh{W}^-_{t,s+r} \leq \wh{W}^-_{t,s} + (\wh{W}_{t,s+r} - \wh{W}_{t,s})^-$, the mean in~\eqref{e:103.8} can be further bounded by
	\begin{equation}
		\label{e:103.9}
		\bbE_{0,0}^{t,0} \big(\wh{W}^-_{t,s} + 1)^2 + \bbE_{0,0}^{t,0} \Big(\big(\wh{W}^-_{t,s} + 1) \big(\wh{W}_{t,s+r} - \wh{W}_{t,s})^- \Big) \,.
	\end{equation}
	
	Recall $\wh{W}_{t,s} := W_{s} - \gamma_{t,s}$. Using the definition, one can check $\gamma_{t,s} \lesssim  \log^+ s$ and $\gamma_{t,s+r} - \gamma_{t,s} \lesssim  \log (1+\frac{r}{s \vee 1})$. Also, $(W_s, W_{s+r})$ under $\bbP_{0,0}^{t,0}$ has the same joint law as $\Big(W_s - \frac{s}{t} W_t , W_{s+r} -  \frac{s+r}{t} W_t\Big)$ under $\bbP_{0,0}$. As $t\to \infty$, the latter tends jointly to $(W_s, W_{s+r})$ in $\bbL^p(\bbP_{0,0})$ for any $p \geq 1$, as can be easily verified e.g.~using time-reversal property of the BM. It then follows that the limit superior as $t \to \infty$ of~\eqref{e:103.9} is at most a constant times
	\begin{align}
		& \bbE_{0,0} W_s^2 +[1 + \log^+ s]^2 + (\bbE_{0,0} |W_s|+ 1 + \log^+ s\big) \Big(\bbE_{0,0} |W_r|+\log \big(1+\frac{r}{s \vee 1}\big)\Big) \\
		&\lesssim s + 1 + (\sqrt{s} + 1) \Big(\sqrt{r} +\log \big(1+\frac{r}{s \vee 1}\big)\Big) \lesssim (s \vee 1) \Big(1 + \sqrt{\frac{r}{s \vee 1}}\Big)\,.
	\end{align}
	This bounds the mean on the left hand side in~\eqref{e:103.7} by 
	\[C t^{-1} \Big(1 + \sqrt{\frac{r}{s \vee 1}}\Big) \, \bbE \big(\varphi(\cN[s,s+r])\big)\]
	 for all $t$ large enough, and, combined with the lower bound of $ct^{-1}$ on $\bbP_{0,0}^{t,0}(\cA_t)$ as given by Lemma~\ref{lem:15}, completes the proof.

	\end{proof}


\section*{Acknowledgments}
We thank L.-P. Arguin and L. Mytnik  for raising the question about almost sure convergence. We are also grateful to L. Mytnik and O. Zeitouni for useful discussions on Theorem~\ref{1st_result}, and to S. Munier for discussions on~\cite{Munier1}. The research of L.H.~was supported by the Deutsche Forschungsgemeinschaft (DFG, German Research Foundation)  through Project-ID 233630050 - TRR 146, through Project-ID 443891315  within SPP 2265, and Project-ID 446173099. The research of O.L.~and T.W.~was supported by the ISF grant no.~2870/21, and by the BSF award 2018330.
\bibliographystyle{abbrv}
\bibliography{BBM3Problems}
\end{document}